\documentclass[12pt,a4paper]{amsart}
\usepackage[utf8]{inputenc}
\usepackage[T1]{fontenc}
\usepackage{amsmath}
\usepackage{amsthm}
\usepackage{amssymb}
\usepackage[abbrev]{amsrefs}
\usepackage{mathrsfs}
\usepackage[dvipsnames]{xcolor}
\usepackage{bm}
\usepackage{enumitem}
\usepackage{hyperref}
\AtBeginDocument{\def\MR#1{}}
\makeatletter
\@namedef{subjclassname@2020}{%
\textup{2020} Mathematics Subject Classification}
\makeatother
\numberwithin{equation}{section}
\allowdisplaybreaks
\newtheorem{thmintro}{}

\newtheorem{theoremintro}[thmintro]{Theorem}
\newtheorem{thm}{}[section]
\newtheorem{theorem}[thm]{Theorem}
\newtheorem{corollary}[thm]{Corollary}
\newtheorem{lemma}[thm]{Lemma}
\newtheorem{proposition}[thm]{Proposition}
\theoremstyle{definition}
\newtheorem{definition}[thm]{Definition}
\newtheorem{question}[thm]{Question}

\newcommand{\abs}[1]{\left\lvert#1\right\rvert}
\newcommand{\norm}[1]{\left\lVert#1\right\rVert}

\newcommand{\enbrace}[1]{\left\lbrace#1\right\rbrace}

\newcommand{\enpar}[1]{\left(#1\right)}
\DeclareMathOperator{\supp}{supp}
\DeclareMathOperator*{\Ave}{Ave}
\newcommand{\Id}{\ensuremath{\mathrm{Id}}}
\newcommand{\wstartop}{\ensuremath{w^*\mbox{--}}}
\newcommand{\VV}{\ensuremath{\mathbb{V}}}
\newcommand{\YY}{\ensuremath{\mathbb{Y}}}
\newcommand{\XX}{\ensuremath{\mathbb{X}}}
\newcommand{\WW}{\ensuremath{\mathbb{W}}}
\newcommand{\NN}{\ensuremath{\mathbb{N}}}
\newcommand{\RR}{\ensuremath{\mathbb{R}}}
\newcommand{\XB}{\ensuremath{\mathcal{X}}}
\newcommand{\YB}{\ensuremath{\mathcal{Y}}}
\newcommand{\FF}{\ensuremath{\mathbb{F}}}
\newcommand{\Pt}{\ensuremath{\mathcal{P}}}
\newcommand{\Ut}{\ensuremath{\mathcal{U}}}
\newcommand{\Cont}{\ensuremath{\mathcal{C}}}
\newcommand{\It}{\ensuremath{\mathcal{I}}}

\newcommand{\Simp}{\ensuremath{\mathcal{S}}}
\newcommand{\Jt}{\ensuremath{\mathcal{J}}}
\newcommand{\SL}{\ensuremath{\mathcal{L}}}
\newcommand{\xx}{\ensuremath{\bm{x}}}
\newcommand{\yy}{\ensuremath{\bm{y}}}
\newcommand{\ee}{\ensuremath{\bm {e}}}
\newcommand{\EB}{\ensuremath{\mathcal{E}}}

\newcommand{\bphi}{\ensuremath{\bm{\phi}}}
\newcommand{\bpsi}{\ensuremath{\bm{\psi}}}
\newcommand{\LL}{\ensuremath{\bm{L}}}
\newcommand{\UU}{\ensuremath{\bm{U}}}
\newcommand{\Ct}{\ensuremath{\mathcal{C}}}
\newcommand{\Rt}{\ensuremath{\mathcal{R}}}

\newcommand{\Mt}{\ensuremath{\mathcal{M}}}
\newcommand{\Nt}{\ensuremath{\mathcal{N}}}

\newcommand{\Anso}{}
\author[J. L. Ansorena]{Jos\'e L. Ansorena}\address{Department of Mathematics and Computer Sciences\\
Universidad de La Rioja\\
Logro\~no 26004\\ Spain}
\email{joseluis.ansorena@unirioja.es}
\author[G. Bello]{Glenier Bello}
\address{Departamento de Matem\'{a}ticas e
Instituto Universitario de Matem\'{a}ticas y Aplicaciones\\
Universidad de Zaragoza\\
50009 Zaragoza\\
Spain}
\email{gbello@unizar.es}
\begin{document}
\title{Mutually non isomorphic mixed-norm Lebesgue spaces}
\begin{abstract}
We prove that for $1\le p,q\le\infty$ the mixed-norm spaces $L_q(L_p)$ are mutually non-isomorphic, with the only exception that $L_q(L_2)$ is isomorphic to $L_q(L_q)$ for all $1<q<\infty$.
\end{abstract}
\thanks{Both authors acknowledge the support of the Spanish Ministry for Science and Innovation under Grant PID2022-138342NB-I00 for \emph{Functional Analysis Techniques in Approximation Theory and Applications (TAFPAA)}. G. Bello has also been partially supported by PID2022-137294NB-I00, DGI-FEDER and by Project E48\_23R, D.G. Arag\'{o}n}
\maketitle
\section{Introduction}\label{sect:intro}\noindent
From a functional analytic point of view, determining whether two given Banach spaces are isomorphic is a primer question already present in Banach's book \cite{Banach1932}. In this setting, mixed-norm spaces came to the forefront in 1973 when Triebel \cite{Triebel1973} proved that, given $1<p$, $q<\infty$, the Besov space $B_{p,q}^\alpha(\RR^d)$ is isomorphic to $\ell_q(\ell_p)$ for any degree of smoothness $\alpha>0$ and any dimension $d\in\NN$. Later on, Lemari\'{e} and Meyer \cite{LemarieMeyer1986} extended the isomorphism to the whole range $p$, $q\in[1,\infty]$. Therefore, the isomorphic classification of Besov spaces $B_{p,q}^\alpha(\RR^d)$ relies on the corresponding classification of matrix spaces $\ell_q(\ell_p)$, $p$, $q\in[1,\infty]$.

Triebel \cite{Triebel1978} said that he learned from Pe{\l}czy\'{n}ski that, for $1<p<\infty$ and $1\le q\le\infty$, different $\ell_q(\ell_p)$-spaces were not isomorphic, and Peetre included this result in his handbook on Besov spaces \cite{Peetre1976}. However, probably because of the unexpected subtlety of the arguments required to complete a proof of Pe{\l}czy\'{n}ski's claim, such a proof was not released until 2011. That year, Cembranos and Mendoza \cite{CembranosMendoza2011} wrote down a proof of Pe{\l}czy\'{n}ski's claim and improved it by including the extreme values of $p$. More precisely, setting $Z_{p,q}=\ell_q(\ell_p)$ as in \cite{BCLT1985}, they proved the following statement.

\begin{theoremintro}[\cite{CembranosMendoza2011}]\label{thm:CM}
Let $1\le p,q,r,s\le\infty$. The spaces $Z_{p,q}$ and $Z_{s,r}$ are isomorphic if and only if $(p,q)=(r,s)$.
\end{theoremintro}

We refer the reader to \cite{AlbiacAnsorena2015} for an extension of Theorem~\ref{thm:CM} to the quasi-Banach setting, that is, to the case when $0<p,q\le\infty$, and to \cite{AlbiacAnsorena2016b} for a detailed exposition of this topic.

As Besov spaces over compact spaces are concerned, it is known that
\[
B_{p,q}^\alpha\enpar{[0,1]^d)} \simeq B_{p,q}:=\enpar{\oplus_{n=1}^\infty \ell_p^n}_{\ell_q}
\]
for any $p$, $q\in(0,\infty]$, any degree of smoothness $\alpha>0$ and any dimension $d\in\NN$ (see \cite{DeVorePopov1988}*{Corollary 5.3} and \cite{AlbiacAnsorena2017}*{Theorem 4.3}). The isomorphic classification of mixed-norm spaces $B_{p,q}$ was established in \cite{AlbiacAnsorena2017} for the whole range $0<p\le\infty$ and $0< q\le\infty$. Focussing on the locally convex setting, we have the following.

\begin{theoremintro}[\cite{AlbiacAnsorena2017}*{Proposition 2.11}]\label{thm:AA}
Let $1\le p,q,r,s\le\infty$. The spaces $B_{p,q}$ and $B_{s,r}$ are isomorphic if and only if $(p,q)=(r,s)$ or $1<q=s<\infty$ and $\{p,r\}=\{2,q\}$.
\end{theoremintro}

The natural way to pursue this line of research is to consider Lebesgue function spaces, that is, to address the isomorphic classification of mixed-norm Lebesque spaces $L_q(L_p)$ for $1\le p,q\le\infty$. In this article, we complete this task in full by achieving the following result.

\begin{theoremintro}\label{thm:Main}
Let $1\le p,q,r,s\le\infty$. The spaces $L_q(L_p)$ and $L_s(L_r)$ are isomorphic if and only if either $(p,q)=(r,s)$ or $1<q=s<\infty$ and $\{p,r\}=\{2,q\}$.
\end{theoremintro}

To contextualize our work, we briefly discuss which spaces can be told apart by making the most of previous study of the geometry of mixed-norm Lebesgue spaces. In 1985, Raynaud characterized the spaces $\ell_r$ that embed into $L_q(L_p)$ when $1\le p\le q<\infty$ (see \cite{Raynaud1985}*{Theorem 1}) and the spaces $Z_{r,s}$ that embed into $L_q(L_p)$ when $1\le q\le p<\infty$ (see \cite{Raynaud1985}*{Theorem 3}). Using mutual embeddability and duality, along with some basic Banach-space theory knowledge, allows for distinguishing nearly all non-obviously isomorphic separable spaces of the form $L_q(L_p)$ through specific casework. Two disjoint classes of such spaces, within each of which isomorphic classification is not easily achievable using this method, remain. The first class is
\[
\Rt_1=\{L_1(L_p):1<p\le2\},
\]
and the second is
\[
\Rt_2=\{L_p(L_1):1<p\le2\}.
\]
Besides, the aforementioned Raynaud's theorems allow us to distinguish a space from $\Rt_1$ and another from $\Rt_2$.

The intuitive reason the classes $\Rt_1$ and $\Rt_2$ resist this approach is that two spaces within one of them are embedded into each other, and their duals are embedded into each other as well. To make progress in each of these classes, a natural line of inquiry is to refine Raynaud's theorems by characterizing their complemented $\ell_q$-subspaces. In this regard, Corollary~\ref{cor:lpcomplemented} will serve our purposes.

Once separable mixed-norm Lebesgue spaces are classified by isomorphism, one should take care of studying the spaces $L_p(L_q)$ when $\max\{p,q\}=\infty$. Since these spaces have received limited attention in this context, we now enter a more uncertain area. By studying the finite-dimensional spaces $\ell_q^n(\ell_p^n)$, $n\in\NN$, which sufficiently describe the spaces $L_q(L_p)$, we see that if $L_q(L_p)\simeq L_r(L_s)$, then also $L_{q'}(L_{p'})\simeq L_{r'}(L_{s'})$
{\Anso (this follows from Proposition~\ref{prop:motivation}) below.}
Now, there is a single case that cannot be simplified to a known one. Namely, telling apart the spaces $L_\infty(L_1)$ and $L_1(L_\infty)$. If these spaces were isomorphic, then Proposition~\ref{prop:motivation} would give that the spaces $\ell_1^n(\ell_\infty^n)$ and $\ell_\infty^n(\ell_1^n)$, $n\in\NN$, would be uniformly complemented in subsequence of each other, which contradicts the findings of Bourgain et al.\@ \cite{BCLT1985}.

Although a more direct roadmap for deriving Theorem~\ref{thm:Main} from Proposition~\ref{prop:motivation} and Corollary~\ref{cor:lpcomplemented} can be devised, we have chosen to write the article in a way that the results obtained do not depend on the advances achieved previously. Our motives for doing so are two-sided. On the one hand, we will use our techniques to provide new proofs of Theorems~\ref{thm:CM} and \ref{thm:AA}. On the other hand, the straight route towards Theorem~\ref{thm:Main} would leave some side results, which could be interesting in their own right, out. Notwithstanding, aiming to fit the expectations of the readers interested solely in Theorem~\ref{thm:Main}, we have outlined an alternative proof of it in Section~\ref{sec:appendix}.

The organization of the paper is as follows. In Section~\ref{sec:prelim}, we introduce some notation and preliminary results that will be employed later on. In particular, to substantiate the guess that finite-dimensional spaces $\ell_q^n(\ell_p^n)$ retain enough information about mixed-norm Lebesgue spaces $L_q(L_p)$, we define a suitable partial ordering on the set of indices $[1,\infty]^2$. This order relation will allow us, with the help of some ultrafilter results by Levy and Raynaud \cite{LevyRaynaud1984}, to properly state and prove the above mentioned Proposition~\ref{prop:motivation}.

In Section~\ref{sec:unconditional}, we achieve some results concerning the embeddability of sequence spaces into mixed-norm Lebesgue spaces, which will be an essential element of our proof of Theorem~\ref{thm:Main}. Then, in Section~\ref{sec:proof}, we complete this proof, and we give new proofs of Theorem~\ref{thm:CM} and Theorem~\ref{thm:AA}. Finally, in Section~\ref{sec:Raynaud}, we use the machinery we developed to prove Theorem~\ref{thm:Main} to study when $\ell_p$ complementably embeds into $L_s(L_r)$, as well as to revisit some results of Raynaud \cite{Raynaud1985} concerning the embeddability of $\ell_p$-spaces into mixed-norm Lebesgue spaces.

\section{\texorpdfstring{$L_q(L_p)$}{}-spaces in a nutshell}\label{sec:prelim}\noindent
In this section we single out the basics of the theory of mixed-norm Lebesgue spaces. All the results we record are either known or a ready consequence of standard techniques.

We say that a family $(\XX_\lambda)_{\lambda\in\Lambda}$ of Banach spaces over the real or complex field $\FF$ $C$-\emph{isomorphically embeds} into another family $(\YY_\lambda)_{\lambda\in\Lambda}$, $1\le C<\infty$, and we put
\[
(\XX_\lambda)_{\lambda\in\Lambda} \sqsubseteq_C (\YY_\lambda)_{\lambda\in\Lambda},
\]
if for each $\lambda\in\Lambda$ there is an isomorphic embedding $J_\lambda\colon \XX_\lambda \to \YY_\lambda$ with
\[
\norm{J_\lambda} \norm{J_\lambda^{-1}} \le C.
\]

In the case when the family $(J_\lambda)_{\lambda\in\Lambda}$ is a lifting of a suitable family of projections, i.e., there are linear maps $P_\lambda\colon \YY_\lambda\to \XX_\lambda$ such that $P_\lambda\circ J_\lambda=\Id_{\XX_\lambda}$ and
\[
\norm{J_\lambda} \norm{P_\lambda} \le C, \quad\lambda\in\Lambda,
\]
we say that $(\XX_\lambda)_{\lambda\in\Lambda}$ \emph{$C$-complementably embeds} into $(\YY_\lambda)_{\lambda\in\Lambda}$, and write
\[
(\XX_\lambda)_{\lambda\in\Lambda} \trianglelefteq_C (\YY_\lambda)_{\lambda\in\Lambda}.
\]
If $\XX_\lambda=\XX$ (resp., $\YY_\lambda=\YY$) for all $\lambda\in\Lambda$, we replace $(\XX_\lambda)_{\lambda\in\Lambda}$ by $\XX$ (resp., $(\YY_\lambda)_{\lambda\in\Lambda}$ by $\YY$) in the above terminology. If the constant $C$ is irrelevant, we simply drop it from the notation. The symbol $\XX\simeq \YY$ means that the Banach spaces $\XX$ and $\YY$ are isomorphic.

If a Banach space $\XX$ complementably embeds into another Banach space $\YY$, then we can decompose the space $\YY$ in two pieces. Namely, there is a third Banach space $\WW$ such that
\[
\YY\simeq \XX\oplus \WW.
\]
Thus, complemented embeddability is often more convenient than isomorphic embeddability for studying the isomorphic theory of Banach spaces. Note that the former condition is stronger than the latter. We will use the following consequence of the principle of local reflexivity.

\begin{proposition}[see \cite{JRZ71}*{Theorem~3.3}]\label{prop:CPLR}
Let $\VV$ and $\XX$ be Banach spaces with $\dim(\VV)<\infty$. Let $C\ge 1$ be such that $\VV \trianglelefteq_C \XX^{**}$. Then, $\VV \trianglelefteq_D \XX$ for all $D>C$.
\end{proposition}
\subsection{\texorpdfstring{$L_q(L_p)$}{}-spaces as vector-valued \texorpdfstring{$L_q$}{}-spaces}
\emph{Function norms} as defined in the nowadays classical manual \cite{BennettSharpley1988} are a natural framework to define mixed-norm Lebesgue spaces. Given a $\sigma$-finite measure space $(\Omega,\Sigma,\mu)$ and a Banach space $(\XX, \norm{\cdot})$, we denote by $L_0(\mu,\XX)$ the vector space consisting of all strongly measurable functions from $\Omega$ to $\XX$. As usual, we identify functions that differ on a null set. Similarly, we identify measurable sets $A$, $B\in\Omega$ such that $\mu(A\ominus B)=0$. The support of $f\in L_0(\mu,\XX)$ is the set
\[
\supp_\mu(f)=f^{-1}(\XX\setminus\{0\}).
\]
Let $L_0^+(\mu)$ be the positive cone of $L_0(\mu)=L_0(\mu,\FF)$. If $\rho\colon L_0^+(\mu) \to [0,\infty]$ is a function norm, then the vector-valued function space
\[
\LL_\rho(\XX)=\enbrace{f\in L_0(\mu,\XX) \colon \norm{f}_{\LL_\rho(\XX)}= \rho\enpar{\norm{f}_\XX}<\infty}
\]
is a Banach space. Note that the scalar-valued space $\LL_\rho=\LL_\rho(\FF)$ is a Banach lattice.

We focus on vector-valued spaces that arise from the function norm $\norm{\cdot}_q$, $1\le q\le \infty$. Set
\[
L_q(\mu,\XX)=\enbrace{f\in L_0(\mu,\XX) \colon \left( \int_\Omega \norm{f}_\XX^q d\mu\right)^{1/q}<\infty},
\]
with the usual modification if $q=\infty$. If $\mu$ is the Lebesgue measure on $[0,1]$ we set $L_q(\mu,\XX)=L_q(\XX)$. If $\mu$ is the counting measure on a set $\Nt$, we set $L_q(\mu,\XX)=\ell_q(\Nt,\XX)$. Set also
\[
\ell_q(\XX)=\ell_q(\NN,\XX), \quad \ell_q^n(\XX)=\ell_q(\NN_n,\XX), \quad n\in\NN,
\]
where $\NN_n=\{1,\dots,n\}$.

The well-known isomorphisms between scalar-valued $L_q$-spaces over different measure spaces (see e.g.\@ \cite{AlbiacKalton2016}*{Chapter 6}) pass with minor modifications to Banach-valued spaces. In fact, while $\ell_q(\XX)$ and $L_q(\XX)$ are not isomorphic in general,
\begin{equation}\label{eq:Lqmu01}
L_q(\mu,\XX) \simeq L_q(\XX)
\end{equation}
for every $1\le q <\infty$, every Banach space $\XX$, and every $\sigma$-finite, separable and not purely atomic measure space $\mu$. In particular, if the measures $\mu_1$ and $\mu_2$ are $\sigma$-finite, separable and not purely atomic, and $p\in[1,\infty]$ and $q\in[1,\infty)$, then
\[
L_q(\mu_1,L_p(\mu_2))\simeq L_q(L_p).
\]
In this paper we will mainly focus on studying the Banach spaces $L_q(L_p)$, $p$, $q\in[1,\infty]$. In the separable case, that is, when $p$, $q\in[1,\infty)$, we can regard these spaces as scalar-valued spaces built from function norms over $[0,1]^2$. In general, given two absolutely continuous function norms $\rho_i$, $i=1$, $2$, over $\sigma$-finite measure spaces $(\Omega_i,\Sigma_i,\mu_i)$, we can construct a product function norm $\rho_1(\rho_2)$ over the product space $(\Omega_1\times \Omega_2,\Sigma_1\otimes \Sigma_2,\mu_1\otimes \mu_2)$ (see e.g.\@ \cite{AnsorenaBello2022}*{Proposition 11}), and we can identify the lattices $\LL_{\rho_1}(\LL_{\rho_2})$ and $\LL_{\rho_1(\rho_2)}$. As Lebesgue spaces are concerned, given $p$, $q\in[1,\infty)$ and $\sigma$-finite measure spaces $(\Omega_i,\Sigma_i,\mu_i)$, $i=1$, $2$, we can identify the vector-valued space $L_q(\mu_1,L_p(\mu_2,\XX))$ with the mixed-norm space consisting of all functions $f\in L_0(\mu_1\otimes \mu_2,\XX)$ such that
\[
\enpar{\int_{\Omega_1}\enpar{ \int_{\Omega_2} \norm{f(x,y)}_\XX^q d\mu_2(y)}^{p/q} d\mu_1(x)}^{1/p} <\infty.
\]
So, for vectors $f\in L_q(\mu_1,L_p(\mu_2))$ we can consider both $\supp_{\mu_1}(f)\in \Sigma_1$ and $\supp_{\mu_1\otimes\mu_2} (f)\in \Sigma_1\otimes \Sigma_2$. If necessary, to avoid confusion, we will call the latter one the \emph{lattice support} of $f$.

Given $1\le p<\infty$, $L_p(L_p)=L_p([0,1]^2)$, whence
\begin{equation*}
\ell_p(L_p(\XX)) \simeq L_p(\XX) \simeq L_p(L_p(\XX))
\end{equation*}
for any Banach space $\XX$. In constract, the indicator function of the triangle $\{(s,t)\in[0,1]^2 \colon s\le t\}$, which fails to be strongly measurable when regarded as function from $[0,1]$ to $L_\infty$, witnesses that $L_\infty(L_\infty)\subsetneq L_\infty([0,1]^2)$.
{\Anso In fact, these spaces fail to be isomorphic (see \cite{Daher2014}*{Corollary 3.4}). Hence, in the case when $p=q=\infty$ we have two non-isomorphic Lebesgue spaces; namely,}
\[
\ell_\infty\simeq L_\infty \simeq \ell_\infty(L_\infty) \simeq \ell_\infty(\ell_\infty), \quad L_\infty(L_\infty)\simeq L_\infty(\ell_\infty).
\]
{\Anso Besides, $L_\infty\simeq L_\infty(\ell_\infty^n)$ isometrically for all $n\in\NN$.}

Let us write down other isomorphisms that interest us. Pe{\l}czy\'{n}ski \cite{Pel1960} proved that
\begin{equation}\label{eq:LpLpL_2}
L_q(L_2) \simeq L_q(\ell_2) \simeq L_q, \quad 1<q<\infty,
\end{equation}
and
\begin{equation}\label{eq:Bpqlq}
B_{2,q}\simeq \ell_q, \quad 1<q<\infty.
\end{equation}
Both isomorphisms are consequences of the fact that
\begin{equation}\label{eq:FDp2}
(\ell_2^n)_{n=1}^\infty \trianglelefteq (\ell_q^{2^n})_{n=1}^\infty, \quad 1<q<\infty,
\end{equation}
which lead to
\begin{equation}\label{eq:PelAgain}
\ell_2\trianglelefteq L_q, \quad 1<q<\infty.
\end{equation}
As for negative results, we will use that
\begin{equation}\label{eq:LinPel}
\ell_2\not\trianglelefteq L_1
\end{equation}
(see \cite{LinPel1968}*{Theorem 6.1}).

Finally we recall a standard result about Rademacher type and cotype, which we will use even when the spaces are locally nonconvex.

\begin{theorem}\label{thm:TCLqLp}
Let $p$, $q\in(0,\infty]$. If $\max\{p,q\}=\infty$, then $L_q(L_p)$ has no finite cotype and, hence, no type greater than one. Otherwise, the optimal Rademacher type of $L_q(L_p)$ is $\min\{p,q,2\}$, and its optimal cotype is $\max\{p,q,2\}$.
\end{theorem}
\subsection{Simple functions and density}
A function norm $\rho$ over a $\sigma$-finite measure space $(\Omega,\Sigma,\mu)$ is said to be \emph{absolutely continuous} if
\[
\lim_{\mu(E) \to 0} \rho(f\chi_E)=0
\]
for every $f\in L_0^+(\mu)$ with $\rho(f)<\infty$. Given $1\le p<\infty$, $\norm{\cdot}_p$ is absolutely continuous, while $\norm{\cdot}_\infty$ is not unless $\Sigma$ is finite.

If $\rho$ is absolutely continuous, then the dominated convergence theorem holds in $\LL_\rho$, that is, $\lim_n \rho(f_n-f)=0$ provided that the sequence $(f_n)_{n=1}^\infty$ in $L_0(\mu)$ converges to $f$ $\mu$-a.\@e.\@ and $\rho(\sup_n f_n)<\infty$.

The proof of the following result concerning the density of the space $\Simp(\mu,\VV)$ consisting of all simple functions with values in a given vector space $\VV\subseteq\XX$ is standard.

\begin{lemma}\label{lem:density}
Let $\rho$ be an absolutely continuous function norm, and let $\XX$ be a Banach space. If $\VV$ is a dense subspace of $\XX$, then $\Simp(\mu,\VV)$ is dense in $\LL_\rho(\XX)$.
\end{lemma}

We point out that, if $\mu$ is a finite measure over an infinite $\sigma$-algebra, then simple functions are dense in $L_\infty(\mu)$ despite $\norm{\cdot}_\infty$ is not absolutely continuous. However, $\Simp(\mu,\XX)$ fails to be dense in $L_\infty(\mu,\XX)$ as long as $\XX$ is infinite-dimensional (see e.g.\@ \cite{DiestelUhl1977}).
\subsection{Finite dimensional structure}
Given Banach spaces $\XX$ and $\YY$, a bounded linear map $T\colon\XX\to\YY$, and a measure space $(\Omega,\Sigma,\mu)$ we consider the mapping
\[
T_\mu\colon L_0(\mu,\XX) \to L_0(\mu,\YY), \quad f\mapsto T\circ f.
\]
For any function norm $\rho$ over $\mu$, the linear map $T_\mu$ restricts to a linear map from $\LL_\rho(\XX)$ to $\LL_\rho(\YY)$ whose norm is bounded by $\norm{T}$.

The isomorphim \eqref{eq:Lqmu01} partially relies on the posibility of averaging in Banach-valued Lebesgue spaces. Given a partition $\Pt$ of $[0,1]$ into Borel sets there is a natural isometric embedding
\begin{equation}\label{eq:EmblqLq}
J_\Pt=J_\Pt[\XX] \colon \ell_q(\Pt,\XX) \to L_q(\XX), \quad (x_A)_{A\in\Pt} \mapsto \sum_{A\in\Pt} x_A \frac{\chi_A}{\abs{A}^{1/q}}.
\end{equation}
We also consider the averaging map
\begin{equation}\label{proj:EmblqLq}
Q_{\Pt} =Q_{\Pt}[\XX]\colon L_q(\XX) \to \ell_q(\Pt,\XX), \quad f \mapsto\enpar{\abs{A}^{1/q-1} \int_A f}_{A\in\Pt}.
\end{equation}
The operator $Q_{\Pt}$ is a linear contraction, and $Q_{\Pt}\circ J_{\Pt}=\Id_{\ell_q(\Pt,\XX)}$.

Let $1\le q\le\infty$, $\Pt$ be a partition of $[0,1]$ into measurable sets, $\XX$ and $\YY$ be Banach spaces, and $J\colon\YY\to \XX$ and $Q\colon\XX\to\YY$ be bounded linear maps with $Q\circ J=\Id_\YY$. The maps
\begin{align*}
J_{\Pt,\mu}&:=J_\Pt[\XX] \circ J_\mu \colon \ell_q\enpar{\Pt, \YY}\to L_q(\XX), \\
Q_{\Pt,\mu}&:=Q_\Pt[\YY] \circ Q_\mu \colon L_q(\XX) \to \ell_q\enpar{\Pt, \YY},
\end{align*}
where $J_{\Pt}$ and $Q_{\Pt}$ are defined as in \eqref{eq:EmblqLq} and \eqref{proj:EmblqLq}, respectively, satisfy $ Q_{\Pt,\mu} \circ J_{\Pt,\mu}=\Id_{ \ell_q\enpar{\Pt, \XX}}$ and $\norm{Q_{\Pt,\mu}} \norm{J_{\Pt,\mu}} \le \norm{J} \norm{Q}$. Consequently, the following holds.

\begin{lemma}\label{lem:AnsoNew}
Let $1\le q\le \infty$, and $\XX$ and $\YY$ be Banach spaces with $\XX\trianglelefteq \YY$. Then
$\ell_q(\XX)\trianglelefteq L_q(\YY)$.
\end{lemma}

Let $1\le p\le \infty$ and suppose that $\XX$ is an $\SL_p$-space, so that there are $C\in[1,\infty)$, a directed set $\It$, and, for each $i\in \It$, $k_i\in\NN$ and linear maps $J_i\colon \ell_p^{k_i}\to \XX$ and $Q_i\colon \XX\to \ell_p^{k_i}$ such that
\[
Q_i \circ J_i=\Id_{\ell_p^{k_i}}, \quad \norm{J_i}\norm{Q_i} \le C.
\]

Assume that $1\le q<\infty$. Bearing in mind Lemma~\ref{lem:density}, the maps $(J_i)_{\Pt,\mu}$ and $(Q_i)_{\Pt,\mu}$, where $i$ runs over $\It$ and $\Pt$ runs over all finite partitions of $[0,1]$ into measurable sets, witness that the finite dimensional structures of $\ell_q(\XX)$ and $L_q(\XX)$ coincide provided that $q<\infty$. To be precise, we have the following result. Before stating it, we clarify that a net $(A_\lambda)_{\lambda=\Lambda}$ of subsets of a containing set $\Omega$ will be said to be non-decreasing if it is relative to the inclusion ordering, that is, $A_\lambda\subseteq A_\mu$ provided that $\lambda\le\mu$.

\begin{lemma}\label{lem:LqLpvxlqlp}
Let $p$, $q\in [1,\infty]$ and $\XX$ be an $\SL_p$-space. Suppose that $q<\infty$, or $q=\infty$ and $\XX=L_\infty$. Then, there are a directed set $\Lambda$ and, for each $\lambda\in\Lambda$, $k_\lambda\in\NN$ and linear maps
\[
I_\lambda\colon \ell_q^{k_\lambda}(\ell_p^{k_\lambda})\to L_q(\XX), \quad P_\lambda\colon L_q(\XX) \to \ell_q^{k_\lambda}(\ell_p^{k_\lambda})
\]
such that
\begin{itemize}[leftmargin=*]
\item $\sup_\lambda \norm{I_\lambda} \norm{P_\lambda} <\infty$;
\item $P_\lambda\circ I_\lambda =\Id_{ \ell_q^{n_\lambda}(\ell_p^{n_\lambda})}$ for all $\lambda\in\Lambda$; and
\item $\enpar{ I_\lambda(\ell_q^{k_\lambda}(\ell_p^{k_\lambda}))}_{\lambda\in\Lambda}$ is a non-decreasing net whose union is dense in $L_q(\XX)$.
\end{itemize}
\end{lemma}

\begin{proof}
\Anso{
Only the case $q=\infty$ and $\XX=L_\infty$, which asserts that $\YY:=L_\infty(L_\infty)$ is an $\SL_\infty$-space, is in doubt. In this case, we still have a directed set $\It$ and, for each $i\in \It$, $k_i\in\NN$ and linear maps
\[
J_i\colon \YY_i:=L_\infty(\ell_\infty^{k_i})\to \YY, \quad Q_i\colon \YY\to\YY_i
\]
such that
$Q_i \circ J_i=\Id_{\YY_i}$, $\sup_{i\in\It} \norm{J_i}\norm{Q_i} <\infty$, and $(J_i(\YY_i))_{i\in\It}$ is a non-decreasing net whose union is a dense subspace of $\YY$. Since each space $\YY_i$ is isometric to $L_\infty$, we are done.
}
\end{proof}
\subsection{Duality}
The dual function norm of a function norm $\rho$ over a $\sigma$-finite measure space $(\Omega,\Sigma,\mu)$ is the function norm $\rho^*$ defined as
\[
\rho^*(f)=\sup\enbrace{\int_\Omega g f\, d\mu \colon g\in L_0^+(\mu),\, \rho(g) \le 1}, \quad f\in L_0^+(\mu).
\]
Given $1\le p \le\infty$, the dual function norm of $\norm{\cdot}_p$ is $\norm{\cdot}_{p'}$, where $p' \in[1,\infty]$ is the conjugate index defined by
\[
\frac{1}{p'} + \frac{1}{p}=1.
\]

The following result encloses a straightforward generalization of the fact that the dual space of $L_p(\XX)$ is $L_{p'}(\XX^*)$ provided that $1\le p<\infty$ and $\XX^*$ has the Radon--Nikodym property (see \cite{DiestelUhl1977}*{Theorem~1 of Chapter~9}). Since, to the best of our knowledge, it is not available in the literature, we include its proof.

\begin{theorem}\label{thm:weakclose}
Let $\XX$ be a Banach space and $\rho$ be a function norm over a $\sigma$-finite measure space $(\Omega,\Sigma,\mu)$.

\begin{enumerate}[label=(\alph*)]
\item\label{Dual:A}
{\Anso Let $(D_n)_{n=1}^\infty$ be a non-decreasing sequence in $\Sigma$ with $\cup_{n=1}^\infty D_n=\Omega$.}
Let $S$ be the set consisting of all functions $f\in\Simp(\mu,\XX)$ such that $\norm{f}_{\LL_\rho(\XX)}\le 1$ and $\supp_\mu(f)\subseteq D_n$ for some $n\in\NN$. Then, for every $g\in L_0(\mu,\XX^*)$,
\begin{equation}\label{eq:DualityVVFS}
{\Anso \norm{g}_{\LL_{\rho^*}(\XX^*)}}
=N(g):=\sup\enbrace{ \abs{\int_\Omega g(f) \, d\mu} \colon f\in S, \, fg\in L_1(\mu)}.
\end{equation}
\item\label{Dual:B} The map $D\colon \LL_{\rho*}(\XX^*) \to \enpar{\LL_\rho(\XX)}^*$ given by
\[
D(g)= \int_\Omega g(f) \, d\mu, \quad g\in \LL_{\rho*}(\XX^*), \, f\in\LL_\rho(\XX),
\]
is an isometric embedding.
\item\label{Dual:C} If $\rho$ is absolutely continuous and $\XX^*$ has the Radon--Nikodym property, then
$D$ is onto. Besides, if we think of $ \LL_{\rho^*}(\XX^*) $ as the dual space of $\LL_\rho(\XX)$, $\LL_{\rho*}(A,\XX^*)$ is weak*-closed for every $A\in\Sigma$.
\end{enumerate}
\end{theorem}

\begin{proof}
Let us see \ref{Dual:A}. It is clear that $N(g) \le \norm{g}_{\LL_{\rho^*}(\XX^*)}$. To prove the reverse inequality, we first assume that $g$ is countably valued, and we put $g=\sum_{n\in F} x_n^* \chi_{A_n}$, where $(A_n)_{n\in F}$ pairwise disjoint. Fix $t <\norm{g}_{\LL_\rho(\XX^*)}$. There is a simple function $h\in\Omega\to[0,\infty)$ such that $\rho(h)\le 1$ and
\[
\sum_{n\in F} \norm{x_n^*}_{\XX^*} \int_{A_n} h \, d\mu=\int_\Omega \norm{g}_{\XX^*} h \, d\mu >t.
\]
Now, there are $G\subset F$ finite, $m\in\NN$, $(a_n)_{n\in G}$ in $[0,\infty)$ and $(B_n)_{n\in G}$ in $\Sigma$ such that
$a_n< \norm{x_n^*}_{\XX^*}$, $B_n\subseteq A_n\cap D_m$ and $\int_{B_n} h \, d\mu<\infty$ for all $n\in G$, and
\[
\sum_{n\in G} a_n \int_{B_n} h \, d\mu>t.
\]
Let $(x_n)_{n\in G}$ in $\XX$ be such that $\norm{x_n}_\XX \le 1$ and $x_n^*(x_n)\in(a_n,\infty)$ for all $n\in G$. Set $f=h \sum_{n\in G} x_n \chi_{A_n}$. We have $\norm{f}_{\LL_\rho(\XX)}\le \rho(h)$ and
\[
\int_\Omega g(f)\, d\mu=\sum_{n\in G} x_n^*(x_n) \int_{B_n} h \, d\mu>t.
\]

Suppose now that $g$ is bounded and supported on a finite measure set. Then, $\norm{g}_{\LL_{\rho^*}(\XX^*)}<\infty$ and there is a sequence $(g_n)_{n=1}^\infty$ of countably-valued measurable functions such that
\[
\lim_n \norm{g_n-g}_{\LL_{\rho^*}(\XX^*)}=0.
\]
By approximation, $N(g) = \norm{g}_{\LL_\rho(\XX^*)}$.

To prove \eqref{eq:DualityVVFS} in general, we select a non-decreasing sequence $(A_n)_{n=1}^\infty$ of finite measure sets such that $\cup_{n=1}^\infty A_n=\Omega$, and $g$ is bounded on each set $A_n$. We have
\begin{multline*}
\norm{g}_{\LL_{\rho*}(\XX^*)}
=\sup_{n\in\NN} \norm{g\chi_{A_n}}_{\LL_{\rho*}(\XX^*)}
=\sup_{\substack{n\in\NN \\ f\in S }}
\abs{\int_{\Omega} \enpar{g \chi_{A_n}} (f) \, d\mu}\\
=\sup_{\substack{n\in\NN \\ f\in S }} \abs{\int_{\Omega} g\enpar{f\chi_{A_n}} \, d\mu}
\le N(g).
\end{multline*}

Statement \ref{Dual:B} is a ready consequence of \ref{Dual:A}. To prove \ref{Dual:C}, pick a partition $(E_n)_{n=1}^\infty$ of $\Omega$ into finite measure sets. For each $n\in\NN$, let $(E_n,\Sigma_n,\mu_n)$ be the restriction of $(\Omega,\Sigma,\mu)$ to $E_n$. Given a functional $\varphi\colon \LL_\rho(\XX)\to \FF$, the map
\[
\nu_n\colon\Sigma_n\to \XX^*, \quad \nu_n(E)(x)=\varphi(x\chi_E),\quad E\in\Sigma_n, \, x\in\XX,
\]
is well-defined, has total variation bounded by $\rho(\chi_{E_n}) \norm{\varphi}$, and, since $\rho$ is absolutely continuous, is countably additive. If $g_n\in L_1(\mu_n,\XX^*)$ is the Radon--Nikodym derivative of $\nu_n$, then
\begin{equation*}
\varphi(f)=\int_{E_n} g_n(f)\, d\mu, \quad f\in\Simp(\mu_n,\XX).
\end{equation*}

The vector space $\Simp_0(\mu,\XX)$ consisting of all $f\in\Simp(\mu,\XX)$ for which there is $m\in\NN$ such that $\supp_\mu(f)\subseteq \cup_{n=1}^m E_n$ is dense in $\LL_\rho(\XX)$ by Lemma~\ref{lem:density}. The function $g\in L_0(\mu,\XX)$ obtained by glueing the functions $g_n$ satisfies
\begin{equation*}
\varphi(f)=\int_{\Omega} g(f)\, d\mu, \quad f\in\Simp_0(\mu_n,\XX).
\end{equation*}
By \ref{Dual:A}, $\norm{g}_{\LL_{\rho^*}(\XX^*)}\le \norm{\varphi}$. By density, $D(g)=\varphi$.

If $g\in L_0(\mu,\XX)$ is not supported on a given set $A\in\Sigma$, then, by \ref{Dual:A}, there is $f\in \Simp(\mu,\XX)$ such that $g(f)\in L_1(\mu)$, $\int_\Omega g(f)\, d\mu\not=0$, and $\supp_\mu(f)\subseteq \Omega\setminus A$. Consequently, $ \LL_{\rho*}(A,\XX^*)$ is the annihilator of $ \LL_{\rho}(\Omega\setminus A,\XX)$. This proves \ref{Dual:C}.
\end{proof}

Since $\rho^{**}=\rho$, the scalar-valued version of Theorem~\ref{thm:weakclose} gives that if both $\rho$ and $\rho^*$ are absolutely continuous, then $\LL_\rho$ is reflexive and, hence, has the Radon--Nikodym property. This gives the following.

\begin{theorem}\label{thm:dualLpLq}
Let $\rho_1$ and $\rho_2$ be function norms. Suppose that $\rho_1$, $\rho_2$ and $\rho_2^*$ are absolutely continuous. Then
\[
\enpar{\LL_{\rho_1} (\LL_{\rho_2})}^*=\LL_{\rho_1^*} (\LL_{\rho_2^*}).
\]
In particular, if $1<p<\infty$ and $1\le q<\infty$, the dual space of $L_q(L_p)$ is $L_{q'}(L_{p'})$ under the natural dual pairing.
\end{theorem}
\subsection{Complementability in the bidual space}
Within the study of the geometry of Banach spaces, arguments that include taking limit points of bounded sequences often appear. As such arguments rely on the Banach-Alaoglu theorem, it is important to know whether a given Banach space is complemented in its bidual. Note that complemented subspaces inherit this property. Note also that dual spaces have it. Therefore, every complemented subspace of a dual space is complemented in its bidual space.

If $p\in(1,\infty]$, $L_p$ is the dual space of $L_{p'}$. For $p=1$, $L_p=L_1$ is a complemented subspace of the space $\Mt([0,1])$ consisting of all signed measures over $[0,1]$, which, in turn, is the dual of the space $\Cont([0,1])$ consisting of all continuous functions over $[0,1]$. Consequently, $L_p$ is complemented in its bidual for any $p\in[1,\infty]$. Let us record the behaviour of mixed-norm Lebesgue spaces in this regard. Oddly enough, they do not inherit the property of being complemented in their biduals from their components.

\begin{proposition}\label{prop:BDLq1}
Let $1<q\le \infty$. Then
\[
L_q(L_1) \trianglelefteq \enpar{L_{q'}(\Cont([0,1]))}^*.
\]
\end{proposition}

\begin{proof}
Clearly, $L_q(L_1) \trianglelefteq L_q(\Mt([0,1]))$. In turn, $L_q(\Mt([0,1]))$ is complemented in the dual space of $L_{q'}(\Cont([0,1]))$ (see \cite{Daher2013}).
\end{proof}

\begin{theorem}\label{thm:bidual}
Let $p$, $q\in[1,\infty]$, $p\not=q$. Then, $L_q(L_p)$ is complemented in its bidual if and only if $p<\infty$.
\end{theorem}

\begin{proof}
If $1<p<\infty$ and $1<q\le \infty$, then $L_q(L_p)$ is the dual space of $L_{q'}(L_{p'})$ by Theorem~\ref{thm:dualLpLq}. If $1<q\le \infty$ and $p=1$, the result follows from Proposition~\ref{prop:BDLq1}. In turn, if $1<p<\infty$ the result follows from \cite{Emmanuele1996}*{Theorem 4}.

Suppose that $1\le q<\infty$. Since $L_\infty$ contains a copy of $c_0$, \cite{Daher2014} and \cite{Emmanuele1996} give that $L_q(L_\infty)$ is not complemented in its bidual.
\end{proof}
\subsection{Function spaces versus sequence spaces}
The relation $\trianglelefteq$, as well as the relation $\sqsubseteq$, is reflexive and transitive.
{\Anso Despite it not being antisymmetric}
\cite{GowersMaurey1997}, the structure of complemented subspaces of a given Banach space is an essential feature of its geometry. As $L_q(L_p)$-spaces are concerned, the following result will be one of the ingredients of our proof of Theorem~\ref{thm:Main}. Before stating it, we define the order relation on pairs of indices mentioned in Section~\ref{sect:intro}.

\begin{definition}
Let $\alpha=(p,q)$, $\beta=(r,s)\in [1,\infty]^2$. We say that $\alpha\preceq \beta$ if there is $\nu\colon\NN\to\NN$ such that
\[
(\ell_q^n (\ell_p^n))_{n=1}^\infty \trianglelefteq (\ell_s^{\nu(n)} (\ell_r^{\nu(n)}))_{n=1}^\infty.
\]
We say that $\alpha\simeq\beta$ if $\alpha\preceq \beta$ and $\beta\preceq \alpha$.
\end{definition}

For further reference, we record a straightforward application of duality.

\begin{lemma}\label{lem:AnsoAuxDual}
Let $p$, $q$, $r$, $s\in[1,\infty]$ be such that $(p,q)\preceq (r,s)$. Then $(p',q')\preceq (r',s')$.
\end{lemma}

\begin{lemma}\label{lem:AnsoAux}
Let $p$, $q$, $r$, $s\in[1,\infty]$ and $\XX$ be an $\SL_r$-space. Suppose that $s<\infty$, or $s=r=\infty$ and $\XX=L_\infty$. If $(\ell_q^n (\ell_p^n))_{n=1}^\infty \trianglelefteq L_s(\XX)$, then $(p,q)\preceq (r,s)$.
\end{lemma}

\begin{proof}
Let $J_n\colon \ell_q^n (\ell_p^n) \to L_s(\XX)$ and $Q_n\colon L_s(\XX) \to \ell_q^n (\ell_p^n)$, $n\in\NN$, be bounded linear maps such that $Q_n\circ J_n=\Id_{\ell_q^n (\ell_p^n)}$ and
\[
\sup_n \norm{J_n} \norm{Q_n} <\infty.
\]
Let $(I_\lambda, P_\lambda, k_\lambda)_{\lambda\in\Lambda}$ be as in Lemma~\ref{lem:LqLpvxlqlp} with respect to the indices $r$ and $s$. By the small perturbation technique, there are $C\in[1,\infty)$ and
\[
J_n'\colon \ell_q^n (\ell_p^n) \to \cup_{\lambda\in\Lambda} I_\lambda(\ell_s^{k_\lambda}(\ell_r^{k_\lambda})) , \quad Q_n'\colon L_s(\XX) \to \ell_q^n (\ell_p^n), \quad n\in\NN,
\]
such that $Q'_n\circ J'_n=\Id_{\ell_q^n (\ell_p^n)}$ and $\norm{J_n'} \norm{Q_n'} \le C$ for all $n\in\NN$. Since $\ell_q^n (\ell_p^n)$ is finite dimensional, for each $n\in\NN$ there is $\lambda(n)\in\Lambda$ such that
\[
J_n'(\ell_q^n (\ell_p^n))\subseteq I_{\lambda(n)}\enpar{\ell_s^{k_{\lambda(n)}}(\ell_r^{k_{\lambda(n)}})}.
\]
The operators $P_{\lambda(n)}\circ J_n'$ and $ P_n' \circ I_{\lambda(n)} $, $n\in\NN$, witness that $(p,q)\preceq (r,s)$.
\end{proof}

\begin{proposition}\label{prop:motivation}
Given $p$, $q$, $r$, $s\in[1,\infty]$, we consider the following assertions.
\begin{enumerate}[label=(\alph*)]
\item\label{it:A} $L_q(L_p) \trianglelefteq L_s(L_r)$.
\item\label{it:A1} $Z_{p,q} \trianglelefteq Z_{r,s}$.
\item\label{it:A2} $B_{p,q} \trianglelefteq B_{r,s}$.
\item\label{it:B} $Z_{p,q} \trianglelefteq L_s(L_r)$.
\item\label{it:D} $(\ell_q^n (\ell_p^n))_{n=1}^\infty \trianglelefteq L_s(L_r)$.
\item\label{it:C} $(p,q)\preceq (r,s)$.
\end{enumerate}
Then,
\begin{enumerate}[label=(\roman*)]
\item\label{New:it:1} \ref{it:A} implies \ref{it:B}, \ref{it:B} implies \ref{it:D}, \ref{it:A1} implies \ref{it:D}, \ref{it:A2} implies \ref{it:D}, and \ref{it:D} implies \ref{it:C}.
\item If $\max\{r,s\}<\infty$, \ref{it:C} implies \ref{it:A} and $\max\{p,q\}<\infty$.
\end{enumerate}
\end{proposition}

\begin{proof}
$Z_{p,q} \trianglelefteq L_q(L_p)$ by Lemma~\ref{lem:AnsoNew},
and it is clear that
\[
(\ell_q^n (\ell_p^n))_{n=1}^\infty \trianglelefteq B_{p,q} \trianglelefteq Z_{p,q}.
\]
Hence, to complete the proof of \ref{New:it:1} it suffices to show that \ref{it:D} implies \ref{it:C}. In the case when $s<\infty$ or $s=r=\infty$, since $L_r$ is an $\SL_r$-space, the implication holds by Lemma~\ref{lem:AnsoAux}. Suppose that $s=\infty$ (or just that $s>1$) and $1\le r<\infty$. Set $\XX=L_{s'}(L_{r'})$ if $r>1$ and $\XX=L_{s'}(\Ct([0,1]))$ if $r=1$. By Theorem~\ref{thm:dualLpLq} and Proposition~\ref{prop:BDLq1}, dualizing we obtain
\[
(\ell_{q'}^n (\ell_{p'}^n))_{n=1}^\infty \trianglelefteq \XX^{**}.
\]
Consequently, by Proposition~\ref{prop:CPLR},
\[
(\ell_{q'}^n (\ell_{p'}^n))_{n=1}^\infty \trianglelefteq \XX.
\]
Since $\Ct([0,1])$ is an $\SL_\infty$-space, $(q',p')\preceq (r',s')$ by Lemma~\ref{lem:AnsoAux}. Hence, \ref{it:C} holds by Lemma~\ref{lem:AnsoAuxDual}.

Assume that \ref{it:C} holds and that $r$ and $s$ are finite. Set $\WW=L_q(L_p)$ if $\max\{p,q\}<\infty$ and $\WW=L_\infty$ otherwise. Use Theorem~\ref{thm:bidual} to pick a bounded linear projection $Q\colon \WW^{**} \to \WW$.

By Lemma~\ref{lem:LqLpvxlqlp}, there are a directed set $\Lambda$ and, for each $\lambda\in\Lambda$, $n_\lambda\in\NN$, a Banach space $\WW_\lambda$, and linear maps
$P_\lambda \colon \WW \to \WW_\lambda$ and $I_\lambda\colon \WW_\lambda\to \WW$ such that
\begin{itemize}
\item $\sup_\lambda \norm{I_\lambda} <\infty$,
\item $\sup_\lambda \norm{P_\lambda} <\infty$,
\item $\lim_{\lambda\in\Lambda} I_\lambda(P_\lambda(f))=f$ for all $f\in\WW$, and
\item $(\WW_\lambda)_{\lambda\in \Lambda} \trianglelefteq (\ell_s^{n_\lambda} (\ell_r^{n_\lambda}))_{\lambda\in \Lambda}$.
\end{itemize}
Let $\Ut$ be a nonprincipal ultrafilter over $\Lambda$. We have
\[
\XX:=\left( \oplus_{\lambda\in\Lambda} \WW_\lambda \right)_\Ut \trianglelefteq \YY:=\left( \oplus_{\lambda\in\Lambda} \ell_s^{n_\lambda} (\ell_r^{n_\lambda})\right)_\Ut.
\]
The maps
\[
f \mapsto (P_\lambda(f))_{\lambda\in\Lambda}, \quad (f_\lambda)_{\lambda\in\Lambda} \mapsto Q( \wstartop\lim_\Ut I_\lambda (f_\lambda))
\]
witness that $\WW \trianglelefteq \XX$. The ultrapower of $L_s(L_r)$ is isomorphic to a band (and thus a complemented subspace) of $L_s(\mu_1,L_r(\mu_2))$, over some probability spaces $\mu_1$ and $\mu_2$ (see \cite{LevyRaynaud1984}). Therefore, $\YY \trianglelefteq L_s(\mu_1,L_r(\mu_2))$. Consequently, $\WW \trianglelefteq L_s(\mu_1,L_r(\mu_2))$.

If $\WW=L_\infty$ we arise to contradiction. Hence, $p$ and $q$ are finite, and $\WW=L_q(L_p)$. Since $\WW$ is separable, there are separable measures $\nu_1$ and $\nu_2$ obtained as restrictions of $\mu_1$ and $\mu_2$, respectively, such that $\WW \trianglelefteq L_s(\nu_1,L_r(\nu_2))$. Since $L_s(\nu_1,L_r(\nu_2)) \trianglelefteq L_s(L_r)$, \ref{it:A} holds.
\end{proof}
\section{Unconditional basic sequences in \texorpdfstring{$L_q(L_p)$}{}}\label{sec:unconditional}\noindent
A sequence $(\xx_n)_{n=1}^{\infty}$ in a Banach space $\XX$ is a \emph{basic sequence} if it is a Schauder basis of its closed linear span $[\xx_n\colon n\in\NN]$. An \emph{unconditional basic sequence} modelled on an infinite countable set $\Nt$ is a family $(\xx_n)_{n\in\Nt}$ in $\XX$ such that $(\xx_{\pi(n)})_{n=1}^\infty$ is a basic sequence of every rearrangement $\pi$ of $\Nt$. A family $\XB=(\xx_n)_{n\in\Nt}$ in $\XX$ is an unconditional basic sequence if and only if there is a constant $K\in[1,\infty)$ such that
\[
\norm{\sum_{n\in\Nt} a_n \, \xx_n}\le K\norm{\sum_{n\in\Nt} b_n \, \xx_n}, \quad (b_n)_{n\in\Nt}\in c_{00},\, \abs{a_n}\le \abs{b_n},
\]
in which case we say that $\XB$ is $C$-unconditional (see \cite{AlbiacKalton2016}*{Proposition 3.1.3}).

Two families $(\xx_n)_{n\in\Nt}$ and $(\yy_m)_{m\in\Mt}$ in Banach spaces $\XX$ and $\YY$, respectively, are said to be \emph{permutatively $C$-equivalent}, $1\le C<\infty$, if there are a bijection $\pi\colon\Nt\to\Mt$ and an isomorphic embedding
\[
T\colon[\xx_n \colon n\in\Nt]\to \YY
\]
such that $\max\{\norm{T},\norm{T^{-1}}\}\le C$ and $T(\xx_n)=\yy_{\pi(n)}$ for all $n\in\Nt$. If $\XX=\YY$ and $T$ is the restriction of an automorphism $S$ of $\XX$, and $\max\{\norm{S},\norm{S^{-1}}\}\le C$, we say that $\XB$ and $\YB$ are \emph{permutatively $C$-congruent}. If $\Nt=\Mt$ and $\pi$ is the identity map we say that $\XB$ and $\YB$ are \emph{$C$-equivalent} or \emph{$C$-congruent}, respectively. If $\pi$ is the identity map and $T$ is just a bounded linear map with $\norm{T}\le C$, we say that $\XB$ \emph{$C$-dominates} $\YB$. If the constant $C$ is irrelevant we simply drop it from the notation.

A family in a Banach space is said to be $C$-complemented, $1\le C<\infty$, if its closed linear span is. If an unconditional basic sequence $(\xx_n)_{n\in\Nt}$ in a Banach space $\XX$ is complemented, then there are functionals $(\xx_n^*)_{n\in\Nt}$ in $\XX^*$ such that the map
\[
f\mapsto \sum_{n\in\Nt} \xx_n^*(f) \, \xx_n
\]
is a bounded linear projection from $\XX$ onto $[\xx_n \colon n\in\Nt]$. In that case, we say that $\XB^*$ are projecting functionals for $\XB$.

Let $(\Omega, \Sigma,\mu)$ be a measure space and $\XX$ be a Banach space. Given $f$, $g\in L_0(\mu,\XX)$,
the symbol $\norm{f}_\XX\le\norm{g}_\XX$ means that $\norm{f(\omega)}_\XX\le\norm{g(\omega)}_\XX$ $\mu$-a.\@e.\@ $\omega\in\Omega$.

\begin{lemma}\label{lem:RestrictOp}
Let $\rho$ be a function norm over a $\sigma$-finite measure space. Let $\XX$ be a Banach space. Let $\Psi=(\bpsi_n)_{n\in\Nt}$ be an unconditional basic sequence in $\LL_\rho(\XX)$. Let $\Phi=(\bphi_n)_{n=1}^\infty$ be a pairwise disjointly supported sequence in $\LL_\rho(\XX)$. Suppose that $\norm{\bphi_n}_\XX \le \norm{\bpsi_n}_\XX$ for all $n\in\NN$. Then $\Psi$ dominates $\Phi$.
\end{lemma}

\begin{proof}
Suppose that $\Psi$ is $K$-unconditional. For $(a_n)_{n\in\Nt} \in c_{00}(\Nt)$ we have
\begin{align*}
\norm{\sum_{n\in\Nt} a_n \bphi_n}_{\LL_\rho(\XX)}
&=\norm{\sup_{n\in\Nt} \abs{a_n} \norm{\bphi_n}_\XX}_{\LL_\rho}
\le\norm{\sup_{n\in\Nt} \abs{a_n} \norm{\bpsi_n}_\XX}_{\LL_\rho}\\
&\le \norm{\Ave_{\varepsilon=\pm 1}\norm{\sum_{n\in\Nt} \varepsilon_n a_n \bpsi_n}_\XX}_{\LL_\rho}\\
&\le \Ave_{\varepsilon=\pm 1} \norm{\sum_{n\in\Nt} \varepsilon_n a_n \bpsi_n}_{\LL_\rho(\XX)}
\le K \norm{\sum_{n\in\Nt} a_n \bpsi_n}_{\LL_\rho(\XX)}.\qedhere
\end{align*}
\end{proof}

\begin{lemma}[cf.\ \cite{AnsorenaBello2025}*{Lemma~4.6}]\label{lem:NewOperators}
Let $\rho$ be an absolutely continuous function norm over a $\sigma$-finite measure space. Let $\XX$ be a Banach space such that $\XX^*$ has the Radon--Nikodym property. Let $\Psi=(\bpsi_n)_{n\in\Nt}$ be a complemented unconditional basis sequence in $\LL_\rho(\XX)$ with projecting functionals $\Psi^*=(\bpsi_n^*)_{n\in\Nt}$ regarded as functions in $\LL_{\rho^*}(\XX^*)$. Let $\Phi=(\bphi_n)_{n\in\Nt}$ be a pairwise disjointly supported sequence in $\LL_\rho(\XX)$ and $(\bphi_n^*)_{n\in\Nt}$ be a pairwise disjointly supported sequence in $\LL_{\rho^*}(\XX^*)$. Suppose that
\begin{itemize}
\item $\norm{\bphi_n}_\XX \le \norm{\bpsi_n}_\XX$ and $\norm{\bphi_n^*}_{\XX^*} \le \norm{\bpsi_n^*}_{\XX^*}$ for all $n\in\Nt$, and
\item $\inf_{n\in\Nt} \abs{ \bphi_n^*( \bphi_n)}>0$.
\end{itemize}
Then $\Phi$ is a complemented basic sequence equivalent to $\Psi$. Moreover, there are scalars $(\lambda_n)_{n\in\Nt}$ such that $(\lambda_n \bphi_n^*)_{n\in\Nt}$ is a sequence of projecting functionals for $\Phi$.
\end{lemma}

\begin{proof}
Assume without loss of generality that $\Nt=\NN$ and set $\lambda_n=\ \bphi_n^*(\bphi_n)$ for all $n\in\NN$. By Lemma~\ref{lem:RestrictOp}, the operators $U_m\colon \LL_{\rho*}(\XX^*)\to \LL_{\rho*}(\XX^*)$ given for each $m\in\NN$ by
\[
U_m(f^*) = \sum_{n=1}^m \frac{1}{\lambda_n} f^*(\bpsi_n) \, \bphi^*_n
\]
are uniformly bounded. For each $m\in\NN$, $U_m$ is the dual operator of the operator $S_m\colon\LL_\rho(\XX)\to\LL_\rho(\XX)$ given by
\[
S_m(f) = \sum_{n=1}^m \frac{1}{\lambda_n} \bphi^*_n(f) \, \bpsi_n.
\]

{\Anso Another application of Lemma~\ref{lem:RestrictOp} gives a uniformly bounded sequence of operators $R_m\colon\LL_\rho(\XX)\to\LL_\rho(\XX)$, $m\in\NN$, given by
\[
R_m(f) = \sum_{n=1}^m \bpsi^*_n(f) \, \bphi_n.
\]
We have $S_m(\bphi_n)=\bpsi_n$ and $R_m(\bpsi_n)=\bphi_n$ for all $n\in\NN\cap[1,m]$. On the one hand, this implies that $(\bphi_n)_{n=1}^m$ and $(\bpsi_n)_{n=1}^m$ are uniformly equivalent,
which means that $\Phi$ and $\Psi$ are equivalent.
On the other hand, the uniformly bounded family of operators $(T_m)_{m=1}^\infty$ defined by $T_m:=R_m\circ S_m$
satisfies
\[
T_m(f) = \sum_{n=1}^m \frac{1}{\lambda_n} \bphi^*_n(f) \, \bphi_n, \quad f\in\XX, \, m\in\NN.
\]
Since $\rho$ is absolutely continuous, then $\Phi$ is boundedly complete. Thus, $(T_m)_{m=1}^\infty$ converges in the strong operator topology to an operator $T$}.
Since $T$ is a projection onto the closed subspace of $\LL_\rho(\XX)$ spanned by $\Phi$, we are done.
\end{proof}

Many results on the symmetric basic sequence structure of Banach spaces use that symmetric basic sequences are subsymmetric, so passing to subsequences we do not lose information. When dealing with non-symmetric basic sequences $(\xx_n)_{n\in\Nt}$, such that the unit vector system of $\ell_q(\ell_p)$, $p\not=q$, this technique breaks down. Still, we can obtain valuable information passing to subsequences as long as we can ensure that, given a partition $(\Nt_i)_{i\in \It}$ of $\Nt$ into sets of infinite cardinality, the subsequence we extract contains infinitely many indices belonging to each subset $\Nt_i$. The following lemma substantiates this idea.

\begin{lemma}\label{lem:Embedding}
Let $\rho$ be a function quasi-norm over a finite measure space $(\Omega,\Sigma,\mu)$. Let $\XX$ be a Banach space. Let $\Psi=(\bpsi_{i,j})_{(i,j)\in\It\times\Nt}$ be a semi-normalized family in $\LL_\rho(\XX)$, where $\It$ is countable and $\Nt$ is infinite. Suppose that
\[
\inf_{j\in\Nt}\norm{\bpsi_{i,j}}_{L_1(\XX)}=0
\]
for all $i\in\It$. Let $(\varepsilon_n)_{n=1}^\infty$ be a sequence in $(0,\infty)$.
Then, there is a one-to-one map $\alpha\colon \NN\to \It\times\Nt$ and a sequence $(A_n)_{n=1}^\infty$ in $\Sigma$ satisfying the following properties.
\begin{enumerate}[label=(\alph*)]
\item In general, we can choose $(A_n)_{n=1}^\infty$ to be non-increasing. In the particular case when $\rho$ is absolutely continuous, we can choose $(A_n)_{n=1}^\infty$ to be pairwise disjoint.
\item $\norm{\bpsi_{\alpha(n)} - \bpsi_{\alpha(n)} \chi_{A_n}}_{\LL_\rho(\XX)}\le \varepsilon_n$ for all $n\in\NN$.
\item $ \mu(A_n)\le \varepsilon_n$ for all $n\in\NN$.
\item $ (\{i\}\times \Nt)\cap \alpha(\NN)$ is infinite for all $i\in\It$.
\end{enumerate}
\end{lemma}

\begin{proof}
Pick $\pi\colon\NN\to\It$ such that $\pi^{-1}(i)$ is infinite for all $i\in \It$. Let $(\delta_k)_{k=1}^\infty$ in $(0,\infty)$ be such that
\[
\sum_{k=n}^\infty \delta_k \le \varepsilon_n, \quad n\in\NN.
\]
Fix $i\in \NN$. Since a subsequence of $(\norm{\bpsi_{i,j}}_\XX)_{j\in\Nt}$ converges to zero in $L_1$, so in measure, a further subsequence converges to zero almost uniformly. Hence, given $\delta\in(0,\infty)$, the projection on the first coordinate of the set $J(i,\delta)$ consisting of all pairs $(j,B)\in\Nt\times\Sigma$ such that
\begin{itemize}
\item $\rho(\chi_{\Omega\setminus B})\norm{ \norm{\bpsi_{i,j}}_\XX -\norm{\bpsi_{i,j}}_\XX \chi_B}_\infty\le \delta$, and
\item $\mu(B)\le\delta$
\end{itemize}
is infinite. Note that $(j,B)\in J(i,\varepsilon)$ implies $ \norm{\bpsi_{i,j} -\bpsi_{i,j} \chi_A}_{\LL_\rho(\XX)}\le\delta$ for all $A\in\Sigma$ with $B\subseteq A$. We recursively construct $(i_n,j_n,B_n)_{n=1}^\infty$ such that $i_n=\pi(n)$, $(j_n,B_n)\in J(i_n, \delta_n)$, and $j_n>j_{n-1}$. If we set
\[
\alpha(n)=(\alpha_1(n),\alpha_2(n))=(i_n,j_n), \quad A_n=\cup_{k=n}^\infty B_k, \quad n\in\NN,
\]
then we have $\mu(A_n)\le \varepsilon_n$ and $\norm{\bpsi_{\alpha(n)} - \bpsi_{\alpha(n)} \chi_{A_n}}_{\LL_\rho(\XX)}\le \varepsilon_n/2$ for all $m$, $n\in\NN$ with $n\le m$. Since $(A_n)_{n=1}^\infty$ is non increasing, the general assertion holds.

If $\rho$ is absolutely continuous, we can recursively construct an increasing sequence $\beta\colon\NN\to \NN$ such that $\alpha_1(\beta(n))=\pi(n)$ and
\[
\norm{\bpsi_{\alpha(\beta(n))} \chi_{A_{\beta(n+1)}}}_{\LL_\rho(\XX)}\le \frac{\varepsilon_n}{2}, \quad n\in\NN.
\]
Set $D_n=A_{\beta(n)} \setminus A_{\beta(n+1)}$ for all $n\in\NN$. We have $\mu (D_n) \le \varepsilon_n$ and
\[
\norm{\bpsi_{\alpha(\beta(n))} - \bpsi_{\alpha(\beta(n))} \chi_{D_n}}_{\LL_\rho(\XX)}\le \varepsilon_n.
\]
Hence the function $\alpha\circ\beta$ and the sequence $(D_n)_{n=1}^\infty$ witness that the particular assertion holds.
\end{proof}

Let $\EB_{\It}=(\ee_n)_{n\in\It}$ denote the unit vector system on a countable set $\It$. Suppose that a function norm $\lambda$ over $\It$ satisfies $\lambda(\ee_n)=1$ for all $n\in\It$. Then, we say that $\LL_\lambda$ is a \emph{sequence space}. If $\EB_{\It}$ spans the whole space $\LL_\lambda$, so that it is an unconditional basis of the space, we say the $\LL_\lambda$ is \emph{minimal}. It is known that $\LL_\lambda$ is minimal if and only if $\lambda$ is absolutely continuous (see \cite{BennettSharpley1988}*{Theorem 3.13}). So, the dual space of a minimal sequence space is a sequence space. If $\LL_\lambda$ is separable, then $\LL_\lambda$ is minimal (see, e.g., \cite{AnsorenaBello2025}*{Theorem 2.5}).

The following lemma is an easy consequence of the unit vector system of $\ell_p$, $1\le p\le \infty$, being perfectly homogeneous (see \cite{BP1958}).

\begin{lemma}\label{lem:VecPel}
Let $\YY$ be a Banach space, $\UU$ be a minimal sequence space over a countable set $\It$, and $1\le p<\infty$. For each $i\in\It$, let $L_i\colon \ell_{p}\to \UU(\ell_{p})$ and $R_i\colon \UU(\ell_{p})\to \ell_{p}$ denote the canonical $i$th lifting and the canonical $i$th projection, respectively. Let $S\colon \UU(\ell_p) \to \YY$ and $P\colon\YY\to\UU(\ell_p)$ be such that $P\circ S=\Id_{\UU(\ell_p)}$. For each $i\in\It$, let $\XB_i=(\xx_{i,n})_{n=1}^\infty$ and $\XB_i^*=(\xx_{i,n}^*)_{n=1}^\infty$ be disjointly supported sequences in $c_{00}$ such that $\supp(\xx_{i,n})=\supp(\xx_{i,n}^*)$ and
\[
\norm{\xx_{i,n}}_{p}=\norm{\xx_{i,n}^*}_{p'}=\xx_{i,n}^*(\xx_{i,n})=1
\]
for all $n\in\NN$. Then there are maps $S_0\colon \UU(\ell_p) \to \YY$ and $P_0\colon\YY\to\UU(\ell_p)$ such that $P_0\circ S_0=\Id_{\UU(\ell_p)}$, and
\[
S_0(\ee_{i,n})=S(L_i(\xx_{i,n})), \quad P_0^*(\ee_{i,n})=P^*(R_i^*(\xx_{i,n}^*))
\]
for all $i\in\It$ and $n\in\NN$.
\end{lemma}

\begin{proof}
For each $i\in\It$, $\XB_i$ is complemented in $\ell_p$, and $\XB_i^*$ are projecting functionals for $\XB_i$ (see \cite{Pel1960}*{Lemma 1}). Quantitatively, there are linear contractions $T_i\colon \ell_p\to \ell_p$ and $Q_i\colon \ell_p\to \ell_p$ such that $Q_i \circ T_i=\Id_{\ell_p}$, and $T_i(\ee_n)=\xx_{i,n}$ and $Q_i^*(\ee_n)=\xx_{i,n}^*$ for all $n\in\NN$. Let $T\colon \UU(\ell_p) \to \UU(\ell_p)$ and $Q\colon \UU(\ell_p) \to \UU(\ell_p)$ be the linear contractions defined by $R_i\circ T\circ L_i = T_i$ and $R_i\circ Q\circ L_i = T_i$ for all $i\in\It$. The maps $S_0=S\circ T$ and $P_0=Q\circ P$ satisfy the desired properties.
\end{proof}

The following two lemmas are geared to take advantage of the fact that $\ell_p$, $2<p<\infty$, does not embed in $L_1(L_r)$ for any $1\le r <p$.

\begin{lemma}\label{lem:Ulq}
Let $\rho$ be an absolutely continuous function norm over a finite measure space $(\Omega,\Sigma,\mu)$ and $\XX$ be a Banach space with nontrivial cotype.
Let $2< p<\infty$ be such that $\XX$ has cotype smaller than $p$ and $\UU$ be a minimal sequence space over a countable set $\It$. Suposse that $\UU(\ell_p)\sqsubseteq \LL_\rho(\XX)$. Then, there is a disjointly supported sequence
\[
\Phi=(\bphi_{i,n})_{(i,n)\in\It\times\NN}
\]
in $\LL_\rho(\XX)$ equivalent to the unit vector system of $\UU(\ell_p)$. Moreover, if $\XX^*$ has the Radon--Nikodym property and $\UU(\ell_p) \trianglelefteq \LL_\rho(\XX)$, then $\Phi$ is complemented, and there are projecting functionals
\[
(\bphi_{i,n}^*)_{(i,n)\in\It\times\NN}
\]
for $\Phi$ with $\supp(\bphi_{i,n}^*)=\supp(\bphi_{i,n})$ for all $(i,n)\in\It\times\NN$.
\end{lemma}

\begin{proof}
Let $J$ denote the embedding of $\LL_\rho(\XX)$ into $L_1(\mu,\XX)$. For each $i\in\It$, let $L_i\colon \ell_p\to \UU(\ell_p)$ be the canonical $i$th lifting. By assumption, there is an isomorphic embedding $S\colon \UU(\ell_p)\to \LL_\rho(\XX)$.

Fix $i\in\It$. Since $L_1(\mu,\XX)$ inherits from $\XX$ the property of having cotype smaller than $p$, and $\ell_p$ is $\ell_p$-saturated (see \cite{Pel1960}*{Lemma~2}), the bounded linear map
\[
J\circ S\circ L_i\colon \ell_p \to L_1(\mu,\XX)
\]
is not an isomorphic embedding on any infinite dimensional subspace of $\ell_p$. Therefore, there is a pairwise disjointly supported sequence $(\xx_{i,n})_{n=1}^\infty$ in $c_{00}$ such that
\begin{itemize}
\item $\norm{\xx_{i,n}}_p=1$ for all $n\in\NN$, and
\item $\lim_n\norm{ S(L_i(\xx_{i,n}))}_{L_1(\XX)}=0$.
\end{itemize}
Set $\bpsi_{i,n}=S(L_i(\xx_{i,n}))$ for all $i\in\It$ and $n\in\NN$. Put
\[
c=\inf \enbrace{ \norm{S(x)}_{\LL(\rho)} \colon \norm{x}_{\UU(\ell_p)}=1}
\]
and pick $(\varepsilon_{n})_{n=1}^\infty$ in $(0,c/2)$ such that
\[
\sum_{n=1}^\infty \frac{2\varepsilon_n}{c-2\varepsilon_n}<1.
\]
Applying Lemma~\ref{lem:Embedding} gives a one-to-one map $\alpha\colon\NN\to \It\times \NN$ and a pairwise disjoint sequence $(A_n)_{n=1}^\infty$ in $\Sigma$ such that
\begin{itemize}
\item for all $i\in\It$, $\alpha(\NN)$ hits the $i$th row infinitely many times, and
\item $\norm{ \bpsi_{\alpha(n)} -\bpsi_{\alpha(n)} \chi_{A_n} }_{\LL_\rho(\XX)}\le \varepsilon_n$ for all $n\in\NN$.
\end{itemize}
Since the canonical basis of $\ell_p$ is symmetric and perfectly homogeneous, $(\bpsi_{\alpha(n)})_{n=1}^\infty$ is permutatively equivalent to the unit vector system of $\UU(\ell_p)$.

Set $\bphi_n=\bpsi_{\alpha(n)} \chi_{A_n}$ for each $n\in\NN$. Let $(\bphi_n^*)_{n=1}^\infty$ in $(\LL_{\rho}(\XX))^*$ be such that $ \bphi_n^*(\bphi_n )=1$, and $\norm{\bphi_n^*}=1/\norm{\bphi_n}_{\LL_\rho(\XX)}$ for all $n\in\NN$. We have
\[
\sum_{n=1}^\infty \norm{\bphi_n^*} \norm{ \bpsi_{\alpha(n)} -\bphi_n}_{\LL_\rho(\XX)}<1.
\]
By the small perturbation principle (see \cite{AnsorenaBello2025}*{Lemma 3.5}), $(\bpsi_{\alpha(n)})_{n=1}^\infty$ is congruent to $\Phi=(\bphi_n)_{n=1}^\infty$.

If $\Psi$ is complemented, so is $\Phi$. Let $(\bphi_n^*)_{n=1}^\infty$ be projecting functionals for $\Phi$. By Lemma~\ref{lem:NewOperators}, so are $(\bphi_n^*\chi_{B_n})_{n=1}^\infty$, where $B_n=\supp(\bphi_n)\subseteq A_n$.
\end{proof}

\begin{lemma}\label{lem:UlqDisjoint}
Let $\rho$ be an absolutely continuous function norm over a finite measure space $(\Omega,\Sigma,\mu)$ and $\XX$ be a Banach space. Let $1\le p<2$ and let $\UU$ be a minimal sequence space over a countable set $\It$. Suppose that $\XX$ has Rademacher type greater than $p$, $\XX^*$ has the Radon--Nikodym property, and $\UU(\ell_p)\ \trianglelefteq \LL_\rho(\XX)$. Then, there is a disjointly supported sequence
\[
\Phi=(\bphi_{i,j})_{(i,j)\in\It\times\NN}
\]
in $\LL_\rho(\XX)$ equivalent to the unit vector system of $\UU(\ell_p)$. Moreover, there are projecting functionals
\[
\Phi^*=(\bphi_{i,j}^*)_{(i,j)\in\It\times\NN}
\]
for $\Phi$ such that $\supp(\bphi_{i,j}^*)=\supp(\bphi_{i,j})$ for all $i\in\It$ and $n\in\NN$.
\end{lemma}

\begin{proof}
Let $J$ denote the embedding of $\LL_{\rho^*}(\XX^*)$ into $L_1(\mu,\XX^*)$. For each $i\in\It$, let $L_i\colon \ell_{p}\to \UU(\ell_{p})$ and $R_i\colon \UU(\ell_{p})\to \ell_{p}$ be the canonical $i$th lifting and the canonical $i$th projection, respectively. By assumption, there are $S\colon \UU(\ell_{p})\to \LL_{\rho}(\XX)$ and $P\colon \LL_{\rho}(\XX)\to \UU(\ell_{p})$ such that $P\circ S=\Id_{\UU(\ell_p)}$.

Fix $i\in\It$. Since $L_1(\mu,\XX^*)$ has Rademacher cotype smaller than $p'$, the bounded linear map
\[
J\circ P^*\circ R_i^*\colon \ell_{p'} \to L_1(\mu,\XX^*)
\]
fail to be an isomorphic embedding on any infinite dimensional subspace of $\ell_{p'}$. Hence, there is a pairwise disjointly supported sequence $(\xx_{i,n}^*)_{n=1}^\infty$ in $c_{00}$ such that
\begin{itemize}
\item $\norm{\xx_{i,n}^*}_{p'}=1$ for all $n\in\NN$, and
\item $\lim_n \norm{ P^*(R_i^*(\xx_{i,n}^*))}_{L_1(\XX^*)}=0$.
\end{itemize}
Pick $(\xx_{i,n})_{n=1}^\infty$ in $c_{00}$ such that $\supp(x_{i,n})=\supp(x^*_{i,n})$, $\norm{\xx_{i,n}}_p=1$, and $\xx_{n,i}^*(\xx_{n,i})=1$ for all $n\in\NN$. By Lemma~\ref{lem:VecPel}, we can assume that $L_i(\xx_{i,n})=\ee_{i,n}$ and $R_i^*(\xx^*_{i,n})=\ee^*_{i,n}$ for all for all $i\in\It$ and $n\in\NN$. Put $S(\ee_{i,n})=\bpsi_{i,n}$ and $P^*(\ee_{i,n})=\bpsi^*_{i,n}$ for all $i\in\It$ and $n\in\NN$. Pick $(\varepsilon_{n})_{n=1}^\infty$ in $(0,\infty)$ such that
\[
\max\{\norm{S},\norm{P}\} \sum_{n=1}^\infty \varepsilon_n<1.
\]
Applying Lemma~\ref{lem:Embedding} gives a one-to-one map $\alpha\colon\NN\to \It\times \NN$ and a non-increasing sequence $(A_n)_{n=1}^\infty$ in $\Sigma$ such that
\begin{itemize}
\item $ \mu(A_n)\le \varepsilon_n$ for all $n\in\NN$,
\item for all $i\in\It$, $\alpha(\NN)$ hits the $i$th row infinitely many times, and
\item $\norm{ \bpsi_{\alpha(n)}^* -\bpsi_{\alpha(n)}^* \chi_{A_n} }_{\LL_{\rho^*}(\XX^*)}\le \varepsilon_n$ for all $n\in\NN$.
\end{itemize}
By the symmetry of the unit vector system of $\ell_p$, $(\bpsi_{\alpha(n)})_{n=1}^\infty$ is permutatively equivalent to the unit vector system of $\UU(\ell_p)$. By the small perturbation principle (see \cite{AnsorenaBello2025}*{Lemma 3.6}) we can assume that $\supp(\bpsi_{\alpha(n)}^*)\subseteq A_n$ for all $n\in\NN$.

Use absolute continuity to pick $\gamma\colon\NN\to\NN$ increasing such that
\begin{itemize}
\item for all $i\in\It$, $\alpha(\gamma(\NN))$ hits the $i$th row infinitely many times, and
\item $\norm{ \bpsi_{\alpha(\gamma(n))} \chi_{A_{\gamma(n+1)}}}_{\LL_\rho(\XX)}\le \varepsilon_n$ for all $n\in\NN$.
\end{itemize}
Set $B_n=A_{\gamma(n)}$ for all $n\in\NN$. By the symmetry of the unit vector system of $\ell_p$, Theorem~\ref{thm:weakclose}\ref{Dual:C} and the small perturbation principle (see \cite{AnsorenaBello2025}*{Lemma 3.5}), we can assume that there is a bijection $\beta\colon\NN\to\It\times\NN$ such that
\begin{itemize}
\item $\supp(\bpsi_{\beta(n)}^*)\subseteq B_n$ for all $n\in\NN$, and
\item $\supp(\bpsi_{\beta(n)})\subseteq [0,1]\setminus B_{n+1}$ for all $n\in\NN$.
\end{itemize}
By Lemma~\ref{lem:NewOperators}, the sequence $\Phi=(\bpsi_{\beta(n)}\chi_{B_n})_{n=1}^\infty$ is permutatively equivalent to $\Psi$, and $(\bpsi^*_{\beta(n)}\chi_{\Omega\setminus B_{n+1}})_{n=1}^\infty$ are projecting functionals for $\Phi$.
\end{proof}

One of the advantages of dealing with disjointly supported sequences is that it permits to convexify the embeddings we obtain. Recall that given a function space $\LL$ and $t\in[1,\infty)$ its convexification $\LL^{(t)}$ is the Banach lattice consisting of all functions $f$ such that $\abs{f}^t\in\LL$.
{\Anso We refer the reader to \cite{LinTza1979}*{Section~1.d} for the basics of this construction in the setting of abstract Banach lattices. The convexification procedure defines an action of the multiplicative semigroup $[1,\infty)$ on the class of Banach lattices. If we consider $t$-convexified spaces for $0<t<1$, we obtain a group action. The price we pay to obtain this richer structure is widening our scope and going into the study of quasi-Banach lattices. The theory of quasi-Banach lattices dates back to \cite{Kalton1984b} and, as far as we know, this convexification tool was first applied to locally nonconvex spaces in \cite{KMP2003}.}

Suppose that $\UU$ is a sequence space over a countable set $\Nt$, $\LL$ is a function space, and there is an isomorphic embedding $J\colon\UU\to\LL$ such that $J(\ee_n)_{n\in\Nt}$ is disjointly supported in $\LL$. Then,
{\Anso given $0<t<\infty$,}
there is an isomorphic embedding of $\UU^{(t)} \to \LL^{(t)}$ given by
\[
(a_n)_{n\in\Nt} \mapsto \sum_{n\in\Nt} a_n \abs{J(\ee_n)}^{1/t}.
\]
Besides, $\UU$ inherits the lattice convexity and concavity from $\LL$. We record an application of this fact to mixed-norm Lebesgue spaces further reference.

\begin{lemma}\label{lem:CCIneq}
Let $p$, $r$, $s\in[1,\infty]$ and suppose that $L_s(L_r)$ has a lattice disjointly supported sequence equivalent to the unit vector system of $\ell_p$. Then,
\begin{equation}\label{eq:Inequality}
\min\{r,s\}\le p \le\max\{r,s\}.
\end{equation}
\end{lemma}

\begin{proof}
If $p=\infty$, we infer that $L_r(L_s)$ has no nontrivial cotype, whence $\infty\in\{r,s\}$. Otherwise, convexifying, we obtain that $\ell_2$ embeds into $L_{2s/p,2r/p}$. By Theorem~\ref{thm:TCLqLp},
\[
\min\enbrace{\frac{2r}{p},\frac{2s}{p}} \le 2 \le \max\enbrace{\frac{2r}{p},\frac{2s}{p}}.\qedhere
\]
\end{proof}

Note that, given function spaces $\LL_1$ and $\LL_2$ and $0<t<\infty$, then
\[
(\LL_1(\LL_2))^{(t)} =\LL_1^{(t)}\enpar{\LL_2^{(t)}}.
\]
Given $1\le p\le\infty$ and a measure $\mu$, $(L_p(\mu))^{(t)}=L_{pt}(\mu)$.

We are now ready to achieve the wished-for results about the embeddability of sequence spaces into $L_s(L_r)$-spaces.

\begin{proposition}\label{prop:MixedNormLsLr}
Let $2<p <\infty$ and $r$, $s\in[1,\infty)$. Let $\UU$ be a sequence space over a countable set $\It$. Assume that $\UU(\ell_p)\sqsubseteq L_s(L_r)$. Then, \eqref{eq:Inequality} holds. If, moreover, $ r \le s$ and $p\not=r$, then $p=s$ and $\UU$ is lattice isomorphic to $\ell_p(\It)$.
\end{proposition}

\begin{proof}
Applying Lemma~\ref{lem:Ulq} with $\rho=\norm{\cdot}_{L_s(L_r)}$ and $\XX=\FF$ gives a lattice disjointly supported family in $L_s(L_r)$ equivalent to the unit vector system of $\UU(\ell_p)$. By Lemma~\ref{lem:CCIneq}, \eqref{eq:Inequality} holds. To prove the `moreover' part, we apply Lemma~\ref{lem:Ulq} with $\rho=\norm{\cdot}_{s}$ and $\XX=L_{r}$. We obtain a disjointly supported family in $L_{s}(\XX)$ equivalent to the unit vector system of $\UU(\ell_{p})$. We infer that the unit vector system of $\UU(\ell_{p})$ is equivalent to the unit vector system of $\ell_{s}$.
\end{proof}

\begin{corollary}\label{cor:Disjointlp}
Let $1\le r\le s<\infty$. Let $1\le p <\infty$ and $\UU$ be a sequence space over a countable set $\It$. If the unit vector system of $\UU(\ell_p)$ is equivalent to a lattice disjointly supported family of $L_s(L_r)$ then either $p=r$ or $p=s$ and $\UU$ is lattice isomorphic to $\ell_p(\It)$.
\end{corollary}

\begin{proof}
Suppose $p\not=r$. Pick $t$ such that $ts>2$. We have $\UU^{(t)}(\ell_{ts})\sqsubseteq L_{ts}(L_{tp})$. An application of Proposition~\ref{prop:MixedNormLsLr} gives that $tp=ts$ and $\UU^{(t)}$ is lattice isomorphic to $\ell_{tp}$.
\end{proof}

\begin{theorem}\label{thm:MixedNormLsLrComp}
Let $p$, $r$, $s\in[1,\infty)$ with $p\not=2$ and $\UU$ be a sequence space over a countable set $\It$. If $\UU(\ell_p) \trianglelefteq L_s(L_r)$, then inequality \eqref{eq:Inequality} holds. If, moreover, $p<2$ or $s>1$, then either $p=r$ or $p=s$ and $\UU$ is lattice isomorphic to $\ell_p(\It)$.
\end{theorem}

\begin{proof}
Depending on whether $p$ is greater or smaller than $2$, we apply Lemma~\ref{lem:Ulq} or Lemma~\ref{lem:UlqDisjoint} with $\rho=\norm{\cdot}_{L_s(L_r)}$ and $\XX=\FF$. In any case, we obtain a lattice disjointly supported family in $L_s(L_r)$ equivalent to the unit vector system of $\UU(\ell_{p})$. By Lemma~\ref{lem:CCIneq}, \eqref{eq:Inequality} holds.

If $r\le s$, it suffices to apply Corollary~\ref{cor:Disjointlp}. Assume that $s<r$, so that $s\le p< r$.

If $p< 2$, so that $p<\min\{2,r\}$, we apply Lemma~\ref{lem:UlqDisjoint} with $\XX=L_r$ and $\rho=\norm{\cdot}_s$ to obtain a disjointly supported sequence in $\LL_s(\XX)$ equivalent to the unit vector system of $\UU(\ell_p)$. Consequently, $\UU(\ell_p)$ is lattice isomorphic to $\ell_s$.

If $s>1$ and $p>2$, then, by duality, $\UU^*(\ell_{p'}) \trianglelefteq L_{s'}(L_{r'})$. Consequently, $p'=s'$ and $\UU^*$ is lattice isomorphic to $\ell_{p'}(\It)$.
\end{proof}

We close this section with a straightforward consequence of Theorem~\ref{thm:MixedNormLsLrComp}. Given $r$, $s\in[1,\infty]$ we set
\[
\Gamma_{r,s}=\enbrace{p\in[1,\infty] \colon \ell_p \trianglelefteq L_s(L_r) }.
\]
\begin{corollary}\label{cor:lpcomplemented}
Let $r,s\in[1,\infty)$.
\begin{enumerate}[label=(\alph*)]
\item $\Gamma_{1,1}= \{1\}$.
\item If $s>1$ or $1<r\le 2$, then $\Gamma_{r,s}= \{2,r,s\}$.
\item If $r>2$, then $ \{1,2,r\}\subseteq\Gamma_{r,1}\subseteq \{1\} \cup [2,r]$.
\end{enumerate}
\end{corollary}

\begin{proof}
Suppose that $\ell_p \trianglelefteq L_s(L_r)$. If $p<2$, or $p>2$ and $s>1$, then $p\in\{s,r\}$. If $p>2$ and $s=1$, then $2<p\le r$. Bearing in mind \eqref{eq:PelAgain} and \eqref{eq:LinPel}, we are done.
\end{proof}
\section{Isomorphic classification of \texorpdfstring{$L_q(L_p)$}{}-spaces}\label{sec:proof}\noindent
To address the classification we will need to tell apart the unit vector system of $Z_{q,p}$ and $Z_{p,q}$, $p\not=q$. We could approach this task using elementary techniques. Indeed, if there were a one-to-one map $\alpha\colon\NN^2\to\NN^2$ such that the linear mapping given by $\ee_{k,n}\to \ee_{\alpha(k,n)}$, $(k,n)\in\NN^2$, was an isomorphic embedding from $Z_{p,q}$ to $Z_{q,p}$, we would reach an absurdity by analyzing how $\alpha$ transforms rows and columns. Nonetheless, we will use the developed machinery to provide a functional argument that avoids making the computations necessary to complete the proof we have outlined.

\begin{lemma}\label{lem:TABases}
Let $p$, $q$, $r$, $s\in [1,\infty]$. Suppose that there is a constant $C$ such that the unit vector system of $\ell_q^n(\ell_p^n)$, $n\in\NN$, is permutatively $C$-equivalent to a subbasis of the unit vector system of $\ell_s(\ell_r)$. Then, $p=q=r$, $p=q=s$, or $(p,q)=(r,s)$.
\end{lemma}

\begin{proof}
Convexifying, we pass, maintaining the order relation between $r$ and $s$, to indices that satisfy $\min\{p,q,r,s\}>1$. Then, dualizing, we pass, reversing the order relation between $r$ and $s$, to indices that satisfy $\max\{p,q,r,s\}<\infty$. Repeating this process if necessary, we infer that we can assume that $r\le s$ and $\max\{p,q,s\}<\infty$. Convexifying, we can also assume that $\min\{p,q,r,s\}>2$. By Proposition~\ref{prop:motivation}, $Z_{p,q} \trianglelefteq L_s(L_r)$. If $p=q$ then, by Corollary~\ref{cor:lpcomplemented}, $p\in\{r,s\}$. Otherwise, by Proposition~\ref{prop:MixedNormLsLr}, $(p,q)=(r,s)$.
\end{proof}

The following result on uniqueness of structure will be essential for us. We say that a Banach space $\XX$ has a \emph{hereditably unique unconditional basis} if it has a normalized unconditional basis $\XB$ and there is a function $\eta\colon[1,\infty)\times[1,\infty) \to [1, \infty)$ such that any normalized $K$-unconditional basic sequence $C$-complemented in $[\XB_0]$ for some subbasis $\XB_0$ of $\XB$ is permutatively $\eta(C,K)$-equivalent to a subbasis of $\XB_0$. By the Schr\"oder--Bernstein principle for unconditional bases \cite{Wojtowicz1988}, if $\XX$ has hereditably unique unconditional basis, then any complemented subspace of $\XX$ with an unconditional basis has a unique normalized unconditional basis up to equivalence and permutation.

\begin{theorem}[see \cite{BCLT1985}*{Theorem 2.2 and Theorem 4.7}]\label{thm:BCLZ}
$\ell_1(\ell_2)$, $\ell_1(c_0)$ and $c_0(\ell_1)$ have a hereditably unique unconditional basis.
\end{theorem}

\begin{proof}
Although the statements of these theorems of the \emph{Memoir} omit the constants involved in the isomorphisms, the quantification we propose is implicit within their proofs (see \cite{BCLT1985}*{Proposition 2.1 and Proposition 4.1}).
\end{proof}

We shall infer Theorems~\ref{thm:CM}, \ref{thm:AA} and \ref{thm:Main} from the following result concerning finite-dimensional structure.
\begin{theoremintro}\label{thm:MainFD}
Let $1\le p,q,r,s\le\infty$. Then $(p,q) \simeq (r,s)$ if and only if either $(p,q)=(r,s)$ or $1<q=s<\infty$ and $\{p,r\}=\{2,q\}$.
\end{theoremintro}
\begin{proof}
Set $D=(1,\infty)\setminus\{2\}$. In light of \eqref{eq:FDp2}, it suffices to prove that the equivalence relation $\simeq$ is trivial on
\[
\Ct=[1,\infty]^2\setminus \enbrace{(p,p) \colon p\in D}.
\]
We say that a family consisting of pairwise disjoint subsets of $\Ct$ agrees with $\simeq$ if related pairs belong to the same set. Proposition~\ref{prop:motivation} gives that $(p,q)\preceq (r,s)$ and $\max\{r,s\}<\infty$ imply $\max\{p,q\}<\infty$. By Lemma~\ref{lem:AnsoAuxDual}, $(p,q)\preceq (r,s)$ and $\min\{r,s\}>1$ imply $\min\{p,q\}>1$. Hence, the partition $(\Ct_i)_{i=1}^4$ of $\Ct$ given by
\begin{align*}
\Ct_1&=\enbrace{(p,q)\in\Ct \colon \min\{p,q\}>1, \, \max\{p,q\}<\infty},\\
\Ct_2&=\enbrace{(p,q)\in\Ct \colon \min\{p,q\}=1, \, \max\{p,q\}<\infty},\\
\Ct_3&= \enbrace{(p,q)\in\Ct\colon \min\{p,q\}>1, \, \max\{p,q\}=\infty},\\
\Ct_4&= \enbrace{(p,q)\in\Ct \colon \min\{p,q\}=1, \, \max\{p,q\}=\infty}=\enbrace{(1,\infty),(\infty,1)},
\end{align*}
agrees with $\simeq$. Moreover, telling apart the pairs in $\Ct_2$ would lead to telling apart the pairs in $\Ct_3$. As far as $\Ct_4$ is concerned, combining Theorem~\ref{thm:BCLZ} with Lemma~\ref{lem:TABases} gives that neither $(1,\infty)\preceq (\infty,1)$ nor $(\infty,1) \preceq (1,\infty)$. So, it suffices to tell apart the pairs in $\Ct_0:=\Ct_1\cup \Ct_{2}$. To that end, we set
\[
\Jt=\enbrace{A\subseteq[1,\infty) \colon \abs{A} = 2}, \quad \Jt_{0}=\Jt\cup\{\{1\},\{2\}\},
\]
and, for $A\in \Jt_0$,
\[
\Ct_A=\enbrace{(p,q)\in\Ct_0 \colon \{p,q\}=A}.
\]
By Proposition~\ref{prop:motivation} and Corollary~\ref{cor:lpcomplemented}, the partition $(\Ct_A)_{A\in\Jt_0}$ of $\Ct_0$ agrees with $\simeq$. Since $\Ct_{\{1\}}$ and $\Ct_{\{2\}}$ are singletons, it remains to prove that $(s,r)\not\simeq (r,s)$ for all $1\le r <s<\infty$.

Assume by contradiction that $(s,r)\preceq (r,s)$ and $(r,s)\preceq (s,r)$. The former condition gives, by Proposition~\ref{prop:motivation} and Theorem~\ref{thm:MixedNormLsLrComp}, $s=2$. Similarly, if $r>1$ the latter condition would give $r=2$. Consequently, $r=1$. Therefore, $(1,2)\preceq(2,1)$. Since this assertion contradicts the combination of Lemma~\ref{lem:TABases} with Theorem~\ref{thm:BCLZ}, we are done.
\end{proof}

\begin{proof}[Proof of Theorem~\ref{thm:CM}, Theorem~\ref{thm:AA} and Theorem~\ref{thm:Main}]
Note that, since $\ell_p$ is $\ell_p$-saturated (see \cite{Pel1960}*{Lemma 2}), $\ell_p(\ell_2)$ and $\ell_p$ are not isomorphic unless $p=2$. Taking this into account, as well as \eqref{eq:LpLpL_2} and \eqref{eq:Bpqlq}, the result follosws from combining Theorem~\ref{thm:MainFD} with Proposition~\ref{prop:motivation}.
\end{proof}
\section{Embeddings of \texorpdfstring{$\ell_p$}{}-spaces into \texorpdfstring{$L_r(L_s)$}{}-spaces}\label{sec:Raynaud}\noindent
One of the main tools for studying the geometry of a given Banach space $\XX$ is to determine the structure of its basic sequences. Within this area, the premier task is determining for which values of $p$ the sequence space $\ell_p$ is a subspace of $\XX$. In this section, we use the results achieved in Section~\ref{sec:unconditional} to address this question for mixed-norm Lebesgue spaces.

Recall that, given $1\le r\le 2$, $\ell_p \sqsubseteq L_r$ if and only if $p\in [r,2]$ (see \cite{Kadec1958}), while, given $2<r<\infty$, $\ell_p \sqsubseteq L_r$ if and only if $p\in \{2,r\}$ (see \cite{Paley1936}). Also recall that $\ell_p \sqsubseteq L_\infty$ for all $1\le p\le \infty$ (see \cite{Banach1932}).

We start by giving an alternative proof of a result by Raynaud.

\begin{theorem}[\cite{Raynaud1985}*{Theorem 1}]\label{thm:Raynaud1}
Let $1\le r\le s<\infty$. Let $1\le p<\infty$. Then $\ell_p\sqsubseteq L_s(L_r)$ if and only if $\ell_p\sqsubseteq L_s$ or $\ell_p\sqsubseteq L_r$.
\end{theorem}

\begin{proof}
We have to prove the `only if' part. Suppose that $\ell_p\sqsubseteq L_s(L_r)$. If $1\le p \le 2$ then, by Theorem~\ref{thm:TCLqLp}, $p\ge r$. So in this case $\ell_p \sqsubseteq L_r$. Assume that $p>2$. Then, by Proposition~\ref{prop:MixedNormLsLr}, either $p=r$ or $p=s$, and the statement follows.
\end{proof}

In the case when $1\le s<r<\infty$, our techniques give the following:
\begin{enumerate}[label=(\alph*)]
\item\label{it:s<r:p<2} If $1\le p<2$, then $\ell_p\sqsubseteq L_s(L_r)$ if and only if $p\ge s$.
\item\label{it:s<r:p>2} If $p>2$, then $\ell_p\sqsubseteq L_s(L_r)$ implies that $s\le p \le r$.
\end{enumerate}
Indeed, if $p<2$ and $\ell_p\sqsubseteq L_s(L_r)$ then, by Theorem~\ref{thm:TCLqLp}, $\min\{2,s\}\le p$. Consequently, $s<2$ and $s\le p$. This proves \ref{it:s<r:p<2}. In turn, \ref{it:s<r:p>2} follows from Proposition~\ref{prop:MixedNormLsLr}. We point out that \cite{Raynaud1985} overrides \ref{it:s<r:p<2} and \ref{it:s<r:p>2}, and includes an embedding that our methods miss. Namely, if $s\ge 1$ and $\max\{2,s\}<p<r<\infty$, then $\ell_p\sqsubseteq L_s(L_r)$ (see Theorem~\ref{thm:Raynaud3} below).

We turn to the study of complemented embeddability. Recall that $\ell_p \trianglelefteq L_\infty$ if and only if $p=\infty$ (see \cite{Lin1967}), $\ell_p \trianglelefteq L_1$ if and only if $p=1$, and, given $1<r<\infty$, $\ell_p \trianglelefteq L_r$ if and only if $p\in\{2,r\}$ (see \cite{KadPel1962}). Our techniques give the following result, for whose statement and proof we define $A'=\enbrace{p' \colon p\in A}$ for all $A\subseteq[1,\infty]$.

\begin{theorem}\label{thm:CompSpectrum}
Given $r$, $s\in[1,\infty]$, the set $\Gamma_{r,s}$ satisfies the following.
\begin{enumerate}[label=(\alph*)]
\item If $1<s<\infty$, then $\Gamma_{r,s}= \{r,s,2\}$.
\item If $1<r\le 2$, then $\Gamma_{r,1}= \{r,1,2\}$.
\item If $2<r<\infty$, then $ \{r,1,2\}\subseteq\Gamma_{r,1}\subseteq \{1\} \cup [2,r]$.
\item If $1<r<\infty$, then $\Gamma_{r,\infty} = \enpar{\Gamma_{r',1}}'$.
\item $\Gamma_{\infty,1}=\Gamma_{1,\infty}=\{1,\infty\}$.
\end{enumerate}
\end{theorem}

\begin{proof}
Corollary~\ref{cor:lpcomplemented} gives the result in the case when $A:=\{r,s\}\subseteq [1,\infty)$. To settle the case $A\subseteq (1,\infty]$, we will prove that
\[
\Gamma_{r,s} \subseteq \enpar{\Gamma_{r',s'}}', \quad r,s\in(1,\infty], \, (r,s)\not=(1,\infty).
\]
Indeed, if $p\in \Gamma_{r,s}$, then $(p,p)\preceq (r,s)$ by Proposition~\ref{prop:motivation}. By Lemma~\ref{lem:AnsoAuxDual}, $(p',p')\preceq (r',s')$. Applying again Proposition~\ref{prop:motivation} gives $p' \in \Gamma_{r',s'}$.

If $1<s<\infty$ and $r=\infty$, we obtain $\Gamma_{r,s} \subseteq \{2,r,s\}$, and the reverse inclusion is clear. In turn, if $s=\infty$ and $1<r<\infty$, then $\enpar{\Gamma_{r',s'}}'\subseteq \Gamma_{r,s}$ by duality.

The case where $A=\{1,\infty\}$ remains to be dealt with. Suppose by contradiction that $p\in \Gamma_{r,s} \cap(1,\infty)$. By \eqref{eq:FDp2}, Proposition~\ref{prop:motivation} and Lemma~\ref{lem:AnsoAuxDual}, $(2,2)\preceq (\infty,1)$ and $(2,2)\preceq (1,\infty)$. Consequently, there are linear operators $L_n\colon\ell_2^n \to \XX:=c_0(\ell_1)$ and $Q_n \colon \XX\to \ell_2^n$, $n\in\NN$, such that $\sup_n \norm{L_n} \norm{P_n} <\infty$ and $Q_n \circ L_n=\Id_{\ell_2^n}$ for all $n\in\NN$. Following the terminology from \cite{CasKal1999}, this means that $\XX$ is sufficiently Euclidean. However, it is known (see \cite{CasKal1999}*{Corollary 2.5}) that it is not the case.
\end{proof}

We close the paper with the problem that Theorem~\ref{thm:CompSpectrum} leaves open.

\begin{question}
Suposse that $2<p<r<\infty$. Does $\ell_p \trianglelefteq L_1(L_r)$?
\end{question}

\section{Appendix}\label{sec:appendix}\noindent
We give a proof of the isomorphic distinction of the separable $L_q(L_p)$ spaces that are not obviously isomorphic, except within each of the classes $\Rt_1$ and $\Rt_2$ mentioned in Section~\ref{sect:intro}, that relies on results from \cite{Raynaud1985} about the embeddability of $\ell_p$-spaces and $Z_{p,q}$-spaces into mixed-norm Lebesgue spaces. Namely, we will use Theorem~\ref{thm:Raynaud1} and the following one.

\begin{theorem}[\cite{Raynaud1985}*{Theorem 3}]\label{thm:Raynaud3}
Let
{\Anso $1\le s\le r<\infty$}
and $1\le p,q<\infty$. Then $Z_{p,q}$ embeds in $L_s(L_r)$ if and only if one of the following holds:
\begin{itemize}
\item $s\le q\le p\le \max\{r,2\}$ or
\item $s\le q\le r$ and $p=2$.
\end{itemize}
\end{theorem}

We will also use well-known results on the complemented embeddability of $\ell_r$-spaces into $L_p$-spaces (see \eqref{eq:PelAgain} and \eqref{eq:LinPel}), and standard results about $L_q(L_p)$-spaces. Specifically, we will use duality (see Theorem~\ref{thm:dualLpLq}), and Rademacher type and cotype (see Theorem~\ref{thm:TCLqLp}).

\begin{definition}
Let $(r,s),(p,q)\in[1,\infty)^2$. We say that the pairs $(r,s)$ and $(p,q)$ \emph{match} if $(p,q)=(r,s)$ or $q=s\in(1,\infty)$ and $\{p,r\}=\{s,2\}$.
\end{definition}

Clearly, the notion of matching pairs defines an equivalence relation on $[1,\infty)^2$.

We fix $(r,s),(p,q)\in[1,\infty)^2$ and henceforth assume $L_s(L_r)\simeq L_q(L_p)$. We will see that, unless these spaces are non-reflexive and
\[
\max\{p,q,r,s\}\le2,
\]
then these pairs match.

\begin{lemma}\label{lem:appendix1}
If $r=s$ or $p=q$, then $(r,s)$ and $(p,q)$ match.
\end{lemma}

\begin{proof}
Assume that $r=s$ and, thus, $L_s(L_r)\simeq L_s$. If $s=1$, then $L_q$ and $L_p$ are complemented in $L_1$. Hence, $s=r=p=q=1$, and we are done. If $s>1$, then the complementation of $L_q$ and $L_p$ in $L_s$ yields $\{p,q\}\subseteq\{s,2\}$. If $p=q=2$, then $L_s\simeq L_2$. Consequently, $s=2$, and we are done. Therefore, unless the pairs match, $(p,q)=(2,s)$ and $s\ne 2$. We obtain $L_s\simeq L_2(L_s)$. By duality, there is $1<t<2$, such that $L_t(L_2) \simeq L_t \simeq L_2(L_t)$. Hence, $Z_{t,2}\sqsubseteq L_t(L_2)$. Applying \cite{Raynaud1985}*{Theorem 3} we reach an absurdity.
\end{proof}

\begin{lemma}\label{lem:appendix0}
If $q>1$ and $p=2$, then $(r,s)$ and $(p,q)$ match.
\end{lemma}

\begin{proof}
We have $L_q(L_q) \simeq L_q(L_p) \simeq L_s(L_r)$. By Lemma~\ref{lem:appendix1} and transitivity, $(p,q)$ and $(r,s)$ match.
\end{proof}

We will say that $(r,s)$ and $(p,q)$ have the same order type if $s-r$ and $q-p$ are either both nonnegative or both nonpositive.

\begin{lemma}\label{lem:appendix2}
If the pairs $(r,s)$ and $(p,q)$ do not match, then they have the same order type.
\end{lemma}

\begin{proof}
Assume that the conclusion is false. Applying Lemma~\ref{lem:appendix1}, and swapping the pairs if necessary, we infer that $1\le s<r$ and $1\le p<q$. Since $Z_{p,q}\sqsubseteq L_s(L_r)$, applying \cite{Raynaud1985}*{Theorem 3} gives that $s\le q\le r$ and $p=2$. We conclude the proof by applying Lemma~\ref{lem:appendix0}.
\end{proof}

\begin{proposition}
If $L_s(L_r)$ is reflexive, then the pairs $(r,s)$ and $(p,q)$ match.
\end{proposition}

\begin{proof}
Assume by contradiction that the pairs do not match. By Lemma~\ref{lem:appendix2}, they have the same order type. By duality, we may assume that $1<s\le r$ and $1<q\le p$. By Lemma~\ref{lem:appendix0}, $2\notin\{p,r\}$. Applying Theorem~\ref{thm:Raynaud3} twice, once to $(r,s)$ and $(p,q)$ and once to the pairs swapped, we get
\[
s\le q\le p\le \max\{r,2\}, \quad
{\Anso q\le s\le r \le \max\{p,2\}.}
\]
This means that $q=s$, and either $p=r>2$ or $\max\{p,r\}<2$. In the latter case, dualizing and applying Theorem~\ref{thm:Raynaud3} gives $p'=r'$. So, in both cases we obtain that $(r,s)=(p,q)$. This absurdity proves the result.
\end{proof}

\begin{proposition}
If $L_s(L_r)$ is non-reflexive and $\max\{p,q,r,s\}>2$, then the pairs $(r,s)$ and $(p,q)$ match.
\end{proposition}

\begin{proof}
Looking at the cotype of the spaces we infer that
\[
\min\{p,q\}=\min\{r,s\}=1, \quad t:=\max\{p,q\}=\max\{r,s\}>2.
\]
Assume by contradiction that the pairs do not match. By Lemma~\ref{lem:appendix2}, either $(p,q)=(r,s)=(1,t)$ or $(p,q)=(r,s)=(t,1)$. This absurdity puts an end to the proof.
\end{proof}

\begin{proposition}
If $L_s(L_r)$ is non-reflexive and $\max\{p,q,r,s\}\le 2$, then either $p=r=1$ or $q=s=1$.
\end{proposition}

\begin{proof}
If the pairs $(r,s)$ and $(p,q)$ math, then $(r,s)=(p,q)$. Otherwise, $(p,q)$ and $(r,s)$ have the same order type by Lemma~\ref{lem:appendix2}. In both cases, the conclusion holds.
\end{proof}
\section*{Acknowledgement}\noindent
The authors would like to thank an anonymous referee for constructive feedback that helped improve the manuscript.
\section*{Statements and Declarations}\noindent
\subsection*{Conflict of interest}
The authors have no competing interests to declare that are relevant to the content of this article.

\subsection*{Data Availability}
Data sharing not applicable to this article as no datasets were generated or analysed during the current study.

\begin{bibdiv}
\begin{biblist}

\bib{AlbiacAnsorena2015}{article}{
author={Albiac, Fernando},
author={Ansorena, Jos\'{e}~L.},
title={On the mutually non isomorphic {$\ell_p(\ell_q)$} spaces, {II}},
date={2015},
ISSN={0025-584X},
journal={Math. Nachr.},
volume={288},
number={1},
pages={5\ndash 9},
url={https://doi-org/10.1002/mana.201300161},
review={\MR{3310494}},
}

\bib{AlbiacAnsorena2016b}{article}{
author={Albiac, Fernando},
author={Ansorena, Jos\'{e}~L.},
title={The isomorphic classification of {B}esov spaces over {$\mathbb{R}^d$} revisited},
date={2016},
ISSN={2662-2033},
journal={Banach J. Math. Anal.},
volume={10},
number={1},
pages={108\ndash 119},
url={https://doi-org/10.1215/17358787-3336542},
review={\MR{3453526}},
}

\bib{AlbiacAnsorena2017}{article}{
author={Albiac, Fernando},
author={Ansorena, Jos\'{e}~L.},
title={Isomorphic classification of mixed sequence spaces and of {B}esov spaces over {$[0,1]^d$}},
date={2017},
ISSN={0025-584X},
journal={Math. Nachr.},
volume={290},
number={8-9},
pages={1177\ndash 1186},
url={https://doi.org/10.1002/mana.201600236},
review={\MR{3666992}},
}

\bib{AlbiacKalton2016}{book}{
author={Albiac, Fernando},
author={Kalton, Nigel~J.},
title={Topics in {B}anach space theory},
edition={Second Edition},
series={Graduate Texts in Mathematics},
publisher={Springer, [Cham]},
date={2016},
volume={233},
ISBN={978-3-319-31555-3; 978-3-319-31557-7},
url={https://doi.org/10.1007/978-3-319-31557-7},
note={With a foreword by Gilles Godefroy},
review={\MR{3526021}},
}

\bib{AnsorenaBello2022}{article}{
author={Ansorena, Jos\'{e}~L.},
author={Bello, Glenier},
title={Toward an optimal theory of integration for functions taking values in quasi-{B}anach spaces},
date={2022},
ISSN={1578-7303},
journal={Rev. R. Acad. Cienc. Exactas F\'{\i}s. Nat. Ser. A Mat. RACSAM},
volume={116},
number={2},
pages={Paper No. 85, 38},
url={https://doi-org/10.1007/s13398-022-01230-8},
review={\MR{4396837}},
}

\bib{AnsorenaBello2025}{article}{
author={Ansorena, Jos\'{e}~L.},
author={Bello, Glenier},
title={Unconditional basic sequences in function spaces with applications to {O}rlicz spaces},
date={2025},
ISSN={1385-1292},
journal={Positivity},
volume={29},
number={1},
pages={Paper No. 1, 36},
url={https://doi.org/10.1007/s11117-024-01093-w},
review={\MR{4822119}},
}

\bib{Banach1932}{book}{
author={Banach, Stefan},
title={Th\'{e}orie des op\'{e}rations lin\'{e}aires},
publisher={\'{E}ditions Jacques Gabay, Sceaux},
date={1993},
ISBN={2-87647-148-5},
note={Reprint of the 1932 original},
review={\MR{1357166}},
}

\bib{BennettSharpley1988}{book}{
author={Bennett, Colin},
author={Sharpley, Robert},
title={Interpolation of operators},
series={Pure and Applied Mathematics},
publisher={Academic Press, Inc., Boston, MA},
date={1988},
volume={129},
ISBN={0-12-088730-4},
review={\MR{928802}},
}

\bib{BP1958}{article}{
author={Bessaga, C.},
author={Pe{\l}czy\'{n}ski, Aleksander},
title={On bases and unconditional convergence of series in {B}anach
spaces},
date={1958},
ISSN={0039-3223},
journal={Studia Math.},
volume={17},
pages={151\ndash 164},
url={https://doi-org/10.4064/sm-17-2-151-164},
review={\MR{115069}},
}

\bib{BCLT1985}{article}{
author={Bourgain, Jean},
author={Casazza, Peter~G.},
author={Lindenstrauss, Joram},
author={Tzafriri, Lior},
title={Banach spaces with a unique unconditional basis, up to
permutation},
date={1985},
ISSN={0065-9266},
journal={Mem. Amer. Math. Soc.},
volume={54},
number={322},
pages={iv+111},
url={https://doi-org/10.1090/memo/0322},
review={\MR{782647}},
}

\bib{CasKal1999}{article}{
author={Casazza, Peter~G.},
author={Kalton, Nigel~J.},
title={Uniqueness of unconditional bases in {$c_0$}-products},
date={1999},
ISSN={0039-3223},
journal={Studia Math.},
volume={133},
number={3},
pages={275\ndash 294},
review={\MR{1687211}},
}

\bib{CembranosMendoza2011}{article}{
author={Cembranos, Pilar},
author={Mendoza, Jos\'{e}},
title={On the mutually non isomorphic {$\ell_p(\ell_q)$} spaces},
date={2011},
ISSN={0025-584X},
journal={Math. Nachr.},
volume={284},
number={16},
pages={2013\ndash 2023},
url={https://doi-org/10.1002/mana.201010056},
review={\MR{2844675}},
}

\bib{Daher2013}{article}{
author={Daher, Mohammad},
title={{$L^p(G,X^\ast)$} comme sous-espace compl\'{e}ment\'{e} de
{$L^q(G,X)^\ast$}},
date={2013},
ISSN={0010-1354,1730-6302},
journal={Colloq. Math.},
volume={131},
number={2},
pages={273\ndash 286},
url={https://doi.org/10.4064/cm131-2-9},
review={\MR{3092456}},
}

\bib{Daher2014}{article}{
author={Daher, Mohammad},
title={Une remarque sur les sous-espaces compl\'{e}ment\'{e}s de {$VB^p(\mu,X)$}},
date={2014},
ISSN={1631-073X,1778-3569},
journal={C. R. Math. Acad. Sci. Paris},
volume={352},
number={1},
pages={43\ndash 49},
url={https://doi.org/10.1016/j.crma.2013.10.036},
review={\MR{3150767}},
}

\bib{DeVorePopov1988}{article}{
author={DeVore, Ronald~A.},
author={Popov, Vasil~A.},
title={Interpolation of {B}esov spaces},
date={1988},
ISSN={0002-9947,1088-6850},
journal={Trans. Amer. Math. Soc.},
volume={305},
number={1},
pages={397\ndash 414},
url={https://doi.org/10.2307/2001060},
review={\MR{920166}},
}

\bib{DiestelUhl1977}{book}{
author={Diestel, J.},
author={Uhl, J.~J., Jr.},
title={Vector measures},
series={Mathematical Surveys},
publisher={American Mathematical Society, Providence, RI},
date={1977},
volume={No. 15},
note={With a foreword by B. J. Pettis},
review={\MR{453964}},
}

\bib{Emmanuele1996}{article}{
author={Emmanuele, G.},
title={Remarks on the complementability of spaces of {B}ochner integrable functions in spaces of vector measures},
date={1996},
ISSN={0010-2628,1213-7243},
journal={Comment. Math. Univ. Carolin.},
volume={37},
number={2},
pages={217\ndash 228},
review={\MR{1398997}},
}

\bib{GowersMaurey1997}{article}{
author={Gowers, William~T.},
author={Maurey, Bernard},
title={{B}anach spaces with small spaces of operators},
date={1997},
ISSN={0025-5831},
journal={Math. Ann.},
volume={307},
number={4},
pages={543\ndash 568},
url={https://doi-org/10.1007/s002080050050},
review={\MR{1464131}},
}

\bib{JRZ71}{article}{
author={Johnson, W.~B.},
author={Rosenthal, H.~P.},
author={Zippin, M.},
title={On bases, finite dimensional decompositions and weaker structures in {B}anach spaces},
date={1971},
ISSN={0021-2172},
journal={Israel J. Math.},
volume={9},
pages={488\ndash 506},
url={https://doi.org/10.1007/BF02771464},
review={\MR{280983}},
}

\bib{Kadec1958}{article}{
author={Kadets, Mikhail~I.},
title={Linear dimension of the spaces {$L_{p}$} and {$l_{q}$}},
date={1958},
ISSN={0042-1316},
journal={Uspehi Mat. Nauk},
volume={13},
number={6(84)},
pages={95\ndash 98},
review={\MR{101486}},
}

\bib{KadPel1962}{article}{
author={Kadets, Mikhail~I.},
author={Pe{\l}czy{\'n}ski, Aleksander},
title={Bases, lacunary sequences and complemented subspaces in the spaces {$L_{p}$}},
date={1961/1962},
ISSN={0039-3223},
journal={Studia Math.},
volume={21},
pages={161\ndash 176},
review={\MR{0152879}},
}

\bib{Kalton1984b}{article}{
author={Kalton, Nigel~J.},
title={Convexity conditions for nonlocally convex lattices},
date={1984},
ISSN={0017-0895},
journal={Glasgow Math. J.},
volume={25},
number={2},
pages={141\ndash 152},
url={https://doi-org/10.1017/S0017089500005553},
review={\MR{752808}},
}

\bib{KMP2003}{article}{
author={Kami\'nska, A.},
author={Maligranda, L.},
author={Persson, L.~E.},
title={Indices, convexity and concavity of {C}alder\'on-{L}ozanovskii spaces},
date={2003},
ISSN={0025-5521,1903-1807},
journal={Math. Scand.},
volume={92},
number={1},
pages={141\ndash 160},
url={https://doi.org/10.7146/math.scand.a-14398},
review={\MR{1951450}},
}

\bib{LemarieMeyer1986}{article}{
author={Lemari\'{e}, Pierre~G.},
author={Meyer, Yves},
title={Ondelettes et bases hilbertiennes},
date={1986},
ISSN={0213-2230},
journal={Rev. Mat. Iberoamericana},
volume={2},
number={1-2},
pages={1\ndash 18},
url={https://doi-org/10.4171/RMI/22},
review={\MR{864650}},
}

\bib{LevyRaynaud1984}{article}{
author={Levy, M.},
author={Raynaud, Y.},
title={Ultrapuissances de {$L^p(L^q)$}},
date={1984},
journal={Seminar on functional analysis, {P}ubl. {M}ath. {U}niv. {P}aris {VII}},
volume={20},
pages={69\ndash 79},
review={\MR{825306}},
}

\bib{Lin1967}{article}{
author={Lindenstrauss, Joram},
title={On complemented subspaces of {$m$}},
date={1967},
ISSN={0021-2172},
journal={Israel J. Math.},
volume={5},
pages={153\ndash 156},
url={https://doi.org/10.1007/BF02771101},
review={\MR{222616}},
}

\bib{LinPel1968}{article}{
author={Lindenstrauss, Joram},
author={Pe{\l}czy\'{n}ski, Aleksander},
title={Absolutely summing operators in {$L_{p}$}-spaces and their applications},
date={1968},
ISSN={0039-3223},
journal={Studia Math.},
volume={29},
pages={275\ndash 326},
url={https://doi-org/10.4064/sm-29-3-275-326},
review={\MR{0231188}},
}

\bib{LinTza1979}{book}{
author={Lindenstrauss, Joram},
author={Tzafriri, Lior},
title={Classical {B}anach spaces. {II} -- function spaces},
series={Ergebnisse der Mathematik und ihrer Grenzgebiete [Results in Mathematics and Related Areas]},
publisher={Springer-Verlag, Berlin-New York},
date={1979},
volume={97},
ISBN={3-540-08888-1},
review={\MR{540367}},
}

\bib{Paley1936}{article}{
author={Paley, R. E. A.~C.},
title={Some theorems on abstract spaces},
date={1936},
ISSN={0002-9904},
journal={Bull. Amer. Math. Soc.},
volume={42},
number={4},
pages={235\ndash 240},
url={https://doi.org/10.1090/S0002-9904-1936-06277-4},
review={\MR{1563277}},
}

\bib{Peetre1976}{book}{
author={Peetre, Jaak},
title={New thoughts on {B}esov spaces},
series={Duke University Mathematics Series, No. 1},
publisher={Mathematics Department, Duke University, Durham, N.C.},
date={1976},
review={\MR{0461123}},
}

\bib{Pel1960}{article}{
author={Pe{\l}czy\'{n}ski, Aleksander},
title={Projections in certain {B}anach spaces},
date={1960},
ISSN={0039-3223},
journal={Studia Math.},
volume={19},
pages={209\ndash 228},
url={https://doi-org/10.4064/sm-19-2-209-228},
review={\MR{126145}},
}

\bib{Raynaud1985}{article}{
author={Raynaud, Yves},
title={Sous-espaces {$l^r$} et g\'{e}om\'{e}trie des espaces {$L^p(L^q)$} et {$L^\phi$}},
date={1985},
ISSN={0249-6291},
journal={C. R. Acad. Sci. Paris S\'{e}r. I Math.},
volume={301},
number={6},
pages={299\ndash 302},
review={\MR{803223}},
}

\bib{Triebel1978}{book}{
author={Triebel, H.},
title={Interpolation theory, function spaces, differential operators},
publisher={VEB Deutscher Verlag der Wissenschaften, Berlin},
date={1978},
review={\MR{500580}},
}

\bib{Triebel1973}{article}{
author={Triebel, Hans},
title={\"{U}ber die {E}xistenz von {S}chauderbasen in
{S}obolev-{B}esov-{R}\"{a}umen. {I}somorphiebeziehungen},
date={1973},
ISSN={0039-3223},
journal={Studia Math.},
volume={46},
pages={83\ndash 100},
url={https://doi.org/10.4064/sm-46-1-83-100},
review={\MR{338771}},
}

\bib{Wojtowicz1988}{article}{
author={W\'{o}jtowicz, Marek},
title={On the permutative equivalence of unconditional bases in {$F$}-spaces},
date={1988},
ISSN={0208-6573},
journal={Funct. Approx. Comment. Math.},
volume={16},
pages={51\ndash 54},
review={\MR{965366}},
}

\end{biblist}
\end{bibdiv}

\end{document}